\documentclass[article,10pt,letterpaper]{article}

\usepackage[letterpaper,left=1.5in,right=1.5in,top=1.5in,bottom=1.5in]{geometry}

\title{On the uniqueness of $\infty$-categorical enhancements of triangulated categories}
\author{Benjamin Antieau}
\date{\today}

\usepackage[pdfstartview=FitH,
            pdfauthor={Antieau},
            pdftitle={Uniqueness of enhancements},
            colorlinks,
            linkcolor=reference,
            citecolor=citation,
            urlcolor=e-mail,
            backref]{hyperref}

\usepackage{amsmath}
\usepackage{amscd}
\usepackage{amsbsy}
\usepackage{amssymb}
\usepackage{verbatim}
\usepackage{eufrak}
\usepackage{eucal}
\usepackage{microtype}
\usepackage{hyperref}
\usepackage{mathrsfs}
\usepackage{amsthm}
\usepackage{stmaryrd}
\usepackage{xy}
\input xy
\xyoption{all}
\usepackage{tikz}
\usetikzlibrary{matrix,arrows}
\usepackage{dsfont}
\usepackage[T1]{fontenc}

\usepackage{fancyhdr}
\usepackage{color}
\definecolor{todo}{rgb}{1,0,0}
\definecolor{conditional}{rgb}{0,1,0}
\definecolor{e-mail}{rgb}{0,.40,.80}
\definecolor{reference}{rgb}{.20,.60,.22}
\definecolor{mrnumber}{rgb}{.80,.40,0}
\definecolor{citation}{rgb}{0,.40,.80}

\let\oldmarginpar\marginpar
\renewcommand\marginpar[1]{\-\oldmarginpar[\raggedleft\footnotesize #1]%
{\raggedright\footnotesize #1}}

\newcommand{\Ascr}{\mathcal{A}}
\newcommand{\Bscr}{\mathcal{B}}
\newcommand{\Cscr}{\mathcal{C}}
\newcommand{\Dscr}{\mathcal{D}}
\newcommand{\Escr}{\mathcal{E}}
\newcommand{\Fscr}{\mathcal{F}}
\newcommand{\Gscr}{\mathcal{G}}
\newcommand{\Hscr}{\mathcal{H}}
\newcommand{\Kscr}{\mathcal{K}}

\newcommand{\Oscr}{\mathcal{O}}
\newcommand{\Pscr}{\mathcal{P}}

\newcommand{\Rscr}{\mathcal{R}}
\newcommand{\Sscr}{\mathcal{S}}

\newcommand{\Wscr}{\mathcal{W}}

\newcommand{\B}{\mathrm{B}}

\newcommand{\D}{\mathrm{D}}

\renewcommand{\H}{\mathrm{H}}
\newcommand{\K}{\mathrm{K}}
\renewcommand{\L}{\mathrm{L}}

\newcommand{\R}{\mathrm{R}}

\newcommand{\T}{\mathrm{T}}

\renewcommand{\AA}{\mathds{A}}

\newcommand{\EE}{\mathds{E}}
\newcommand{\FF}{\mathds{F}}

\newcommand{\NN}{\mathds{N}}

\newcommand{\QQ}{\mathds{Q}}

\renewcommand{\SS}{\mathds{S}}

\newcommand{\ZZ}{\mathds{Z}}

\newcommand{\gr}{\mathrm{gr}}

\newcommand{\proj}{\mathrm{proj}}

\newcommand{\op}{\mathrm{op}}
\newcommand{\ex}{\mathrm{ex}}
\newcommand{\lex}{\mathrm{lex}}

\newcommand{\cn}{\mathrm{cn}}

\newcommand{\sing}{\mathrm{sing}}

\newcommand{\Sp}{\Sscr\mathrm{p}}

\newcommand{\Op}{\mathrm{Op}}
\newcommand{\Groth}{\mathrm{Groth}}

\newcommand{\cofib}{\mathrm{cofib}}
\newcommand{\SWscr}{\Sscr\Wscr}
\newcommand{\Perfscr}{\Pscr\mathrm{erf}}
\newcommand{\wD}{\widehat{\D}}
\newcommand{\wDscr}{\widehat{\Dscr}}
\newcommand{\cD}{\check{\D}}
\newcommand{\cDscr}{\check{\Dscr}}
\newcommand{\sep}{\mathrm{sep}}
\newcommand{\comp}{\mathrm{comp}}
\newcommand{\Tors}{\mathrm{Tors}}

\newcommand{\heart}{\heartsuit}

\renewcommand{\geq}{\geqslant}
\renewcommand{\leq}{\leqslant}
\newcommand{\Ho}{\mathrm{Ho}}

\newcommand{\Cat}{\mathrm{Cat}}

\newcommand{\PrL}{\mathrm{Pr}^{\mathrm{L}}}

\DeclareMathOperator{\Ext}{Ext}

\newcommand{\THH}{\mathrm{THH}}

\DeclareMathOperator{\Fun}{Fun}

\DeclareMathOperator{\Hom}{Hom}
\newcommand{\Map}{\mathrm{Map}}
\newcommand{\MapSp}{\mathbf{Map}}

\DeclareMathOperator{\PShv}{PShv}
\DeclareMathOperator{\Shv}{Shv}
\newcommand{\Sets}{\mathrm{Sets}}

\newcommand{\Ch}{\mathrm{Ch}}

\newcommand{\Mod}{\mathrm{Mod}}
\newcommand{\QCoh}{\mathrm{QCoh}}
\newcommand{\qc}{\mathrm{qc}}
\newcommand{\Coh}{\mathrm{Coh}}

\newcommand{\Perf}{\mathrm{Perf}}
\newcommand{\Ind}{\mathrm{Ind}}

\newcommand{\Gm}{\mathds{G}_{m}}
\newcommand{\Ga}{\mathds{G}_{a}}

\newcommand{\BP}{\mathrm{BP}}

\DeclareMathOperator*{\colim}{colim}

\newcommand{\et}{\mathrm{\acute{e}t}}

\DeclareMathOperator{\Spec}{Spec}

\newcommand{\lrlarrows}{%
    \tikz[minimum height=0ex]
          \path[->]
             node (a)            {}
       node (b) at (1em,0) {}
         (b.north)  edge (a.north)
               (a.center) edge (b.center)
                 (b.south)  edge (a.south);%
}
\newcommand{\we}{\simeq}
\newcommand{\iso}{\cong}

\theoremstyle{plain}
\newtheorem{theorem}{Theorem}[section]
\newtheorem*{theorem*}{Theorem}
\newtheorem{lemma}[theorem]{Lemma}

\newtheorem{proposition}[theorem]{Proposition}

\newtheorem{conjecture}[theorem]{Conjecture}
\newtheorem{corollary}[theorem]{Corollary}
\newtheorem*{corollary*}{Corollary}

\theoremstyle{plain}
\newtheorem{maintheorem}{Theorem}
\newtheorem{mainmetatheorem}[maintheorem]{Meta Theorem}

\newtheorem{maincorollary}[maintheorem]{Corollary}

\theoremstyle{definition}

\newtheorem{mainexample}[maintheorem]{Example}
\newtheorem{mainremark}[maintheorem]{Remark}

\newtheoremstyle{named}{}{}{\itshape}{}{\bfseries}{.}{.5em}{#1 \thmnote{#3}}
\theoremstyle{named}

\theoremstyle{definition}
\newtheorem{definition}[theorem]{Definition}
\newtheorem{warning}[theorem]{Warning}
\newtheorem{variant}[theorem]{Variant}

\newtheorem{example}[theorem]{Example}
\newtheorem*{example*}{Example}

\newtheorem{question}[theorem]{Question}
\newtheorem*{question*}{Question}

\newtheorem{remark}[theorem]{Remark}

\begin{document}

\maketitle

\begin{abstract}
    \noindent
    We study the problem of when triangulated categories admit unique $\infty$-categorical
    enhancements. Our results use Lurie's theory of prestable $\infty$-categories to
    give conceptual proofs of, and in many cases strengthen, previous work on the subject
    by Lunts--Orlov and Canonaco--Stellari. We also give a wide range of examples involving
    quasi-coherent sheaves, categories of almost modules, and local cohomology to
    illustrate the theory of prestable $\infty$-categories. Finally, we propose
    a theory of stable $n$-categories which would interpolate between
    triangulated categories and stable $\infty$-categories.


    \paragraph{Mathematics Subject Classification 2010.}
    \href{https://mathscinet.ams.org/mathscinet/msc/msc2020.html?t=14Axx&btn=Current}{14A30},
    \href{https://mathscinet.ams.org/mathscinet/msc/msc2020.html?t=14Fxx&btn=Current}{14F08},
    \href{https://mathscinet.ams.org/mathscinet/msc/msc2020.html?t=18Exx&btn=Current}{18E05},
    \href{https://mathscinet.ams.org/mathscinet/msc/msc2020.html?t=18Exx&btn=Current}{18E10},
    \href{https://mathscinet.ams.org/mathscinet/msc/msc2020.html?t=18Gxx&btn=Current}{18G80}.
\end{abstract}

\tableofcontents

\section{Introduction}

This paper is a study of the question of when triangulated categories admit
unique $\infty$-categorical enhancements. Our emphasis is
on exploring to what extent the proofs can be made to rely only on universal
properties. That this is possible is due to J. Lurie's theory of {\em prestable
$\infty$-categories}. We should say at the outset that our results, while
$\infty$-categorical in nature, imply results about dg enhancements as well and recover most
of the major results of the papers of Lunts--Orlov~\cite{lunts-orlov} and
Canonaco--Stellari~\cite{cs-uniqueness}. (See Remark~\ref{mr:cns} for
discussion of the new paper~\cite{cns} of Canonaco--Neeman--Stellari on this subject.) Moreover, there are more
$\infty$-categorical enhancements than dg enhancements in general, so in some
sense what is proved here is stronger.

Suppose that $\Ascr$ is a Grothendieck abelian category. We can attach to $\Ascr$ three
different triangulated categories, the {\em unseparated derived category}
$\check{\D}(\Ascr)$, the {\em separated derived category} $\D(\Ascr)$, and
the {\em completed derived category} $\widehat{\D}(\Ascr)$.

The separated derived category is the familiar triangulated category attached to a
Grothendieck abelian category by inverting all quasi-isomorphisms in the homotopy category
$\K(\Ch(A))$ of chain complexes of objects of $\Ascr$. The unseparated derived category
$\check{\D}(\Ascr)$ was introduced by Krause~\cite{krause-noetherian} and also studied
in~\cite{neeman-injectives,krause-deriving}. It is the
homotopy category of all complexes of injective objects of $\Ascr$ and is often written
$\K(\Ch(\mathrm{Inj}_\Ascr))$. The third triangulated category $\widehat{\D}(\Ascr)$ is less
familiar. It is the homotopy category of the left completion of the standard stable
$\infty$-categorical model of $\D(\Ascr)$ with respect to the standard
$t$-structure. We do not know of a direct construction of $\wD(\Ascr)$ that
starts with the triangulated category $\D(\Ascr)$.

Each of these flavors of the derived category of a Grothendieck abelian category $\Ascr$
admits an $\infty$-categorical enhancement, which we write as $\cDscr(\Ascr)$,
$\Dscr(\Ascr)$, and $\wDscr(\Ascr)$, respectively. In this generality, the results are due to Lurie
thanks to the dg nerve construction. See~\cite[Section~C.5.8]{sag} for $\cD(\Ascr)$
and~\cite[Section~1.3.5]{ha} for $\D(\Ascr)$. For the completed derived category
$\wD(\Ascr)$, we defined it via its enhancement, as in~\cite[Section~C.5.9]{sag}.
In spirit, the enhancement for $\D(\Ascr)$ as a stable model category is classical and goes
back to Joyal (in a 1984 letter to Grothendieck) and Spaltenstein~\cite{spaltenstein}; the
enhancement of $\cD(\Ascr)$ goes back effectively to~\cite{krause-noetherian}. Given the
existence of these enhancements, we wonder about uniqueness.

All three triangulated categories admit natural $t$-structures and the categories of
coconnective objects\footnote{Working homologically, an object $X$ is {\bf coconnective} if,
with respect to the given $t$-structure, $\H_i(X)=0$ for $i>0$; it is {\bf
connective} if $\H_i(X)=0$ for $i<0$.} are all equivalent: $\cD(\Ascr)_{\leq
0}\we\D(\Ascr)_{\leq
0}\we\wD(\Ascr)_{\leq 0}$. The difference between these derived categories then lies in the
categories
$\cD(\Ascr)_{\geq 0}$, $\D(\Ascr)_{\geq 0}$, and $\wD(\Ascr)_{\geq 0}$ of connective objects.
Lurie has developed in~\cite[Appendix~C]{sag} the theory of prestable $\infty$-categories;
these are $\infty$-categories $\Cscr$ such that the homotopy category $h\Cscr$ behaves like
the category of connective objects for a $t$-structure. These $\infty$-categories give a rich
generalization of the theory of abelian categories: the {\em Grothendieck prestable
$\infty$-categories} admit a Gabriel--Popescu theorem, which reduces much of their study
to the $\infty$-categories of the form $\Dscr(R)_{\geq 0}=\Mod_R(\Sp^\cn)$ of
$R$-modules in connective spectra where $R$ is a connective $\EE_1$-algebra.\footnote{Every
dg algebra $R$ has an underlying
$\EE_1$-algebra. The theory of $\EE_1$-algebras is the natural theory of associative algebras
in homotopy theory. See~\cite{shipley} or~\cite{ha} for details.}
Moreover, there are several uniqueness theorems in~\cite{sag}. For example, if $\Ascr$ is a
Grothendieck abelian category, then $\Dscr(\Ascr)_{\geq
0}$ is the unique separated $0$-complicial Grothendieck prestable $\infty$-category with
heart equivalent to $\Ascr$. In the end, our proofs boil down to uniqueness statements such
as these.

An $\infty$-categorical enhancement of a triangulated category $\T$ is a stable
$\infty$-category\footnote{It is important to take $\Cscr$ to be stable
or at least spectrally enriched: we can view any space $X$ as an
$\infty$-groupoid and the homotopy category will be the fundamental groupoid
$\tau_{\leq 1}X$. If $X$ is simply connected, we have thus found an
$\infty$-categorical enhancement for the terminal category, so such
enhancements are far from unique.} $\Cscr$ together with a triangulated equivalence $h\Cscr\we\T$ from
the homotopy category of $\Cscr$, with its canonical triangulated structure, to
$\T$. For remarks on the distinction between $\infty$-categorical enhancements
and dg enhancements, see the discussion around Meta Theorem~\ref{meta:dg},
Section~\ref{sec:meta}, and Section~\ref{sub:nstable}. The dg enhancements model Keller's {\em algebraic
triangulated categories}~\cite{keller}, while stable $\infty$-categories
provide models for Schwede's {\em topological triangulated categories}~\cite{schwede-algebraic}.

Our first theorem gives a partial positive answer to an open question of Canonaco and Stellari;
see~\cite[Question~4.6]{cs-tour}. The following corollary partially
answers~\cite[Question~4.7]{cs-tour} and generalizes several results of~\cite{lunts-orlov,cs-uniqueness}.

Recall that a Grothendieck
abelian category $\Ascr$ is {\bf compactly generated}, or {\bf locally finitely
presented}, if for each $X\in\Ascr$ there is a
collection of compact
(or locally finitely presented) objects\footnote{An object $X$ in a category $\Cscr$ with filtered colimits
is {\bf compact} if $\Hom_{\Cscr}(X,-)$, viewed as a functor $\Cscr\rightarrow\Sets$ commutes with filtered colimits.}
$\{Y_i\}\in\Ascr^\omega$ and a surjection $\oplus Y_i\rightarrow X$.
A Grothendieck abelian category $\Ascr$ is {\bf locally coherent}
if it is compactly generated and $\Ascr^\omega$ is abelian. The latter
condition is equivalent to $\Ascr$ being compactly generated and $\Ascr^\omega$ being closed under finite limits in
$\Ascr$.

\begin{maintheorem}\label{mt:1}
    If $\Ascr$ is a locally coherent Grothendieck abelian category, then the unseparated derived category $\check{\D}(\Ascr)$ admits a unique
    $\infty$-categorical enhancement.
\end{maintheorem}

If additionally $\Ascr$ has enough compact projective objects and each object of
$\Ascr^\omega$ has finite projective dimension, then
$\check{\D}(\Ascr)\we\D(\Ascr)$ (as is implicit in~\cite{krause-noetherian} and
explicit in~\cite[C.5.8.12]{sag}). Thus, in this special case, our result
follows from Theorem~\ref{mt:2} below, which is the $\infty$-categorical
analog of the compactly generated case of~\cite[Theorem~A]{cs-uniqueness}. As far as we are aware, all other
cases are new. One example is $\cD(\QCoh(X))$ where $X=\Spec R$ for a
noetherian but non-regular commutative ring $R$. This is the interesting case since
$\cD(\QCoh(X))$ plays a role in the study of the singularities of $X$ (see for
example~\cite{krause-noetherian} and Section~\ref{sub:sing}).

\begin{maincorollary}\label{maincor:1}
    If $\Ascr$ is a small abelian category, then $\D^b(\Ascr)$ admits a unique
    $\infty$-categorical enhancement.
\end{maincorollary}

The dg categorical version of this statement was a conjecture
of Bondal--Larsen--Lunts~\cite{bondal-larsen-lunts} in the special case when
$\Ascr\we\Coh(X)$ for $X$ a smooth
projective variety over a field. Their conjecture
was proved in~\cite[Theorem~8.13]{lunts-orlov} in fact for
$\D^b(\Coh(X))$ when $X$ is quasi-projective over a field $k$. It was then
generalized
in~\cite[Corollary~7.2]{cs-uniqueness} to $\D^b(\Coh(X))$ when $X$ is
noetherian and has enough locally free sheaves. Our theorem applies equally
well to non-noetherian coherent examples, situations where there are not enough
locally free sheaves, and even to algebraic stacks. For example, it applies to coherent sheaves on all proper noetherian schemes
over $\ZZ$, where it is not currently known if there are enough locally free sheaves.

Our next theorem is about uniqueness of stable $\infty$-categorical enhancements which we assume additionally to be
presentable (the $\infty$-categorical analog of well-generated).

\begin{maintheorem}\label{mt:2}
    If $\Ascr$ is Grothendieck abelian, then $\D(\Ascr)$ admits a unique presentable
    $\infty$-categorical enhancement.
\end{maintheorem}

The differential graded analogue of this theorem was established when
$\D(\Ascr)$ is compactly generated by a set of compact objects of
$\Ascr^\heart$ in~\cite[Theorem~7.5]{lunts-orlov} and in full generality
(and without the presentability caveat) in~\cite{cs-uniqueness}.

\begin{mainremark}
    When $\Ascr$ is a Grothendieck abelian category, all $\infty$-categorical or dg
    enhancements of $\D(\Ascr)$ which arise in practice are manifestly presentable.
    Thus, this additional hypothesis is not a major drawback to the theorem.
    Nevertheless, we explain in Appendix~\ref{sub:appendix} how to
    follow the work of Canonaco and Stellari~\cite{cs-uniqueness} to remove the presentability hypothesis.
    As the purpose of this paper is largely to illustrate the use of prestable
    $\infty$-categories, we do not view this additional generality as the main
    point. Many other statements are given in this paper, for example
    Corollary~\ref{maincor:2} below, which follow from
    Theorem~\ref{mt:2} and would follow, possibly in an easier fashion, from
    the more general result in the appendix. We have found it interesting to
    preserve the arguments flowing from Theorem~\ref{mt:2}.
\end{mainremark}

Given a compactly generated triangulated category $\T$, it turns out that every stable
$\infty$-categorical enhancement of $\T$ is presentable. See
Proposition~\ref{prop:presentable}, which is due to Lurie.

\begin{maincorollary}\label{maincor:2}
    If $\Ascr$ is Grothendieck abelian and $\D(\Ascr)$ is compactly generated, then
    $\D(\Ascr)$ admits a unique $\infty$-categorical enhancement.
\end{maincorollary}

Note that we do not assume that $\D(\Ascr)$ be compactly generated by objects of
the heart $\Ascr$.

\begin{mainexample}
    If $X$ is a quasi-compact scheme with affine diagonal, then $\D(\QCoh(X))$ admits a unique
    stable $\infty$-categorical enhancement. Indeed, in this case
    $$\D(\QCoh(X))\we\D_{\qc}(X)$$ by the argument of~\cite[Corollary~5.5]{bokstedt-neeman}
    (which easily adapts from the separated to the affine diagonal situation). But,
    $\D_{\qc}(X)$ is always compactly generated by~\cite[Theorem~3.1.1]{bondal-vdb} when $X$
    is quasi-compact and quasi-separated. In fact, this example extends to many algebraic stacks by
    work of Hall--Rydh and Hall--Neeman--Rydh. In~\cite{hall-rydh}, various
    conditions are given which guarantee that $\D_{\qc}(X)$ is compactly generated. In these
    cases, if $X$ is quasi-compact and has affine diagonal,
    then~\cite[Theorem~1.2]{hall-neeman-rydh} shows that
    $\D(\QCoh(X))\rightarrow\D_{\qc}(X)$ is an equivalence. For example, $\D_{\qc}(X)$ is
    compactly generated when $X$ is quasi-compact with quasi-finite separated diagonal, in
    which case $\D(\QCoh(X))\we\D_{\qc}(X)$ so that $\D_\qc(X)$ admits a unique
    $\infty$-categorical enhancement by Corollary~\ref{maincor:2}.
\end{mainexample}

\begin{mainremark}
    We can also prove uniqueness of $\infty$-categorical enhancements for the derived
    categories $\D(\Mod_A^a)$ of almost module categories studied
    in~\cite{faltings,gabber-ramero} even though they are
    not generally compactly generated. See Example~\ref{ex:almost}.
\end{mainremark}

Our third main theorem is designed to give a criteria for unique enhancements
of small stable $\infty$-categories. We write $\Dscr^b(\Ascr)$ and $\Perfscr(X)$ for the
natural $\infty$-categorical enhancements of $\D^b(\Ascr)$ and $\Perf(X)$ when $\Ascr$ is a
small abelian category and $X$ is a scheme.

Let $\Dscr^b(\ZZ)_{\geq 0}\subseteq\Dscr^b(\ZZ)$ be the full subcategory of connective
objects. By definition, this a prestable $\infty$-category. Additionally, we can recover
$\Dscr^b(\ZZ)$ by formally inverting the suspension, or translation, functor
$\Dscr^b(\ZZ)_{\geq 0}\xrightarrow{\Sigma}\Dscr^b(\ZZ)_{\geq 0}$. That is, there is an
equivalence $$\colim\left(\Dscr^b(\ZZ)_{\geq 0}\xrightarrow{\Sigma}\Dscr^b(\ZZ)_{\geq
0}\xrightarrow{\Sigma}\Dscr^b(\ZZ)_{\geq
0}\rightarrow\cdots\right)\we\Dscr^b(\ZZ).$$ In general, given a prestable
$\infty$-category $\Cscr_{\geq 0}$ (which is by definition pointed and has finite colimits), we can form its {\bf Spanier--Whitehead category}
$$\colim\left(\Cscr_{\geq 0}\xrightarrow{\Sigma}\Cscr_{\geq
0}\xrightarrow{\Sigma}\Cscr_{\geq 0}\rightarrow\cdots\right)=\SWscr(\Cscr_{\geq 0}),$$
which is a stable $\infty$-category.
Another example is if $X$ is a quasi-compact and quasi-separated scheme. In this case,
$$\Perfscr(X)_{\geq 0}=\Perfscr(X)\cap\Dscr_\qc(X)_{\geq 0}$$ is a prestable
$\infty$-category and its Spanier--Whitehead category is $\SWscr(\Perfscr(X)_{\geq
0})\we\Perfscr(X)$ since every perfect complex on a quasi-compact and quasi-separated scheme
is bounded below. Note however that $\Perfscr(X)_{\geq 0}$ is not generally the connective
part of a $t$-structure on $\Perfscr(X)$.

Let $\Cscr_{\geq 0}$ be a prestable $\infty$-category. We let $\Cscr^\heart_{\geq
0}\subseteq\Cscr_{\geq 0}$ be the full subcategory of $0$-truncated
objects.\footnote{Recall that in an $\infty$-category $\Dscr$, an object $Y$ is
$0$-truncated if the mapping space $\Map_\Cscr(X,Y)$ is $0$-truncated for each
$X\in\Dscr$. Equivalently, $\pi_i\Map_\Cscr(X,Y)=0$ for all $X$ and all $i>0$.
Finally, this condition is equivalent to saying that $\Map_\Cscr(X,Y)$ is
homotopy equivalent to a discrete topological space for all $X$.}
In general, $\Cscr^\heart_{\geq 0}$ is simply an additive category. It need not be abelian
or even have cokernels. We say that $\Cscr_{\geq 0}$ is {\bf $0$-complicial} if
for every object $X\in\Cscr_{\geq 0}$ there is an object
$Y\in\Cscr^\heart_{\geq 0}$ and a map $Y\xrightarrow{u} X$ such that the cofiber of $u$, computed in
$\SWscr(\Cscr_{\geq 0})$, is in $\Cscr_{\geq 1}\we\Cscr_{\geq 0}[1]\subseteq\SWscr(\Cscr_{\geq 0})$.
(Inside the large Spanier--Whitehead category $\SWscr(\Ind(\Cscr)_{\geq 0})$, which has a
$t$-structure with connective part $\Ind(\Cscr_{\geq 0})$, the condition on the cofiber of $u$ is
equivalent to saying that $\pi_0\cofib(u)=0$ or equivalently that $\pi_0(u)$ is surjective.) For example, if $\Ascr$ is a small abelian category, then $\Dscr^b(\Ascr)_{\geq 0}$ is
$0$-complicial. If $X=\Spec R$ is an affine scheme, then $\Perfscr(X)_{\geq 0}$ is
$0$-complicial.

\begin{maintheorem}\label{mt:3}
    Let $\Cscr_{\geq 0}$ be a small prestable $\infty$-category. If $\Cscr_{\geq 0}$ is
    $0$-complicial, then the triangulated category $h\SWscr(\Cscr_{\geq 0})$ admits a unique $\infty$-categorical
    enhancement.
\end{maintheorem}

In the case of $\Ext$-finite triangulated categories over a field, Muro has existence and uniqueness
results for projective modules over certain so-called basic algebras in~\cite{muro-first,muro-enhanced}.

Let $X$ be a quasi-compact and quasi-separated scheme.
We say that $X$ is $0$-complicial if the small prestable $\infty$-category
$\Perfscr(X)_{\geq 0}=\Perfscr(X)\cap\Dscr_{\qc}(X)_{\geq 0}$ is $0$-complicial.
Note that $\Perf(X)_{\geq 0}\cap\Dscr_{\qc}(X)^\heart\subseteq\Perfscr(X)_{\geq
0}^\heart$, but that in general we expect that this inclusion is strict. If $X$ is quasi-compact with
affine diagonal and has enough
locally free sheaves or more generally enough perfect quasi-coherent sheaves (meaning
that $\QCoh(X)$ is generated by perfect quasi-coherent sheaves), then it is $0$-complicial.

\begin{maincorollary}\label{maincor:3}
    If $X$ is quasi-compact, quasi-separated, and $0$-complicial, then $\Perf(X)$ admits a unique
    $\infty$-categorical enhancement.
\end{maincorollary}

Lunts and Orlov proved in~\cite[Theorem~7.9]{lunts-orlov} that if $X$ is
quasi-projective over a field, then $\Perf(X)$ has a unique dg enhancement.
In~\cite[Proposotion~6.10]{cs-uniqueness}, Canonaco and Stellari proved that
$\Perf(X)$ has a unique dg enhancement whenever $X$ is a noetherian
concentrated stack (i.e., $\Perf(X)\subseteq\D_{\qc}(X)^\omega$)
with quasi-finite affine diagonal and enough perfect quasi-coherent sheaves.
Our result for example removes the noetherianity hypotheses from these theorems.

There is a related corollary, of which Corollary~\ref{maincor:3} is a special
case when $X$ additionally has affine diagonal.

\begin{maincorollary}\label{maincor:4}
    If $\Ascr$ is a Grothendieck abelian category such that $\D(\Ascr)$ is compactly
    generated and $\Dscr(\Ascr)_{\geq 0}\cap\Dscr(\Ascr)^\omega$ is $0$-complicial, then
    $\D(\Ascr)^\omega$ admits a unique $\infty$-categorical enhancement.
\end{maincorollary}

We also have the following corollary, a special case of which appears as
Proposition~\ref{prop:ba} and is needed to establish
Theorems~\ref{mt:1} and~\ref{mt:2}. If $\Ascr$ is an abelian category, let
$\D^-(\Ascr)$ and $\D^+(\Ascr)$ denote the bounded below and bounded above
derived categories of $\Ascr$, respectively.

\begin{maincorollary}\label{maincor:5}
    If $\Ascr$ is a small abelian category, then
    $\D^-(\Ascr)$ and $\D^+(\Ascr)$ admit unique $\infty$-categorical enhancements.
\end{maincorollary}

As far as we can see, there are no antecedents to this result in the
literature. Corollary~\ref{maincor:5} gives an answer to a variant
of~\cite[Question~4.7]{cs-tour}. They ask if $\D(\QCoh(X))^\kappa$ admits a
unique dg enhancement when $X$ is an algebraic stack and $\kappa$ is
sufficiently large. Let $\Ascr$ be a Grothendieck abelian category. For $\kappa$ sufficiently large, $\Ascr^\kappa$ is an
abelian category and Krause showed in the main theorem of~\cite{krause-deriving} 
that $\D(\Ascr^\kappa)\we\D(\Ascr)^\kappa$. It follows from the corollary that
$\D(\Ascr)^{\kappa,-}$, the category of bounded below objects of
$\D(\Ascr)^{\kappa}$, admits a unique $\infty$-categorical enhancement, and
similarly for the bounded above derived category.

\begin{mainremark}
    The theorems and corollaries above apply to triangulated categories such as $\D(\QCoh_Z(X))$
    or $\Perf_Z(X)$, where $\QCoh_Z(X)$ is the Grothendieck abelian category of quasi-coherent sheaves
    supported (set theoretically) on a closed subscheme $Z$ of $X$ with
    quasi-compact complement. They also apply to the twisted versions $\D_{\qc}(X,\alpha)$
    where $\alpha\in\H^2_\et(X,\Gm)$ is a (possibly non-torsion) cohomological Brauer class.
    We leave these extensions to the interested reader.
\end{mainremark}

\begin{mainremark}\label{mr:cns}
    Since the first draft of this paper appeared, Canonaco, Neeman, and
    Stellari have made a big advance, settling many of the open
    problems of this paper in~\cite{cns}. Besides giving new proofs of
    Corollaries~\ref{maincor:1} and~\ref{maincor:5}, they prove that $\D(A)$
    admits a unique dg enhancement for $A$ {\em any} abelian category
    (strengthening our Theorem~\ref{mt:2}), they prove
    that $\check{\D}(A)$ and $\hat{\D}(A)$ admits unique dg enhancements for any
    Grothendieck abelian category (strengthening our Theorem~\ref{mt:1}), and they prove that $\D_\qc(X)$ and
    $\Perf(X)$ admit unique dg enhancements for any quasi-compact and
    quasi-separated scheme $X$ (strengthening our Corollary~\ref{maincor:3}).
    Forthcoming work of Jack Hall explains how to identify left complete
    $t$-structure from the triangulated category, so it is possible to give a
    proof of unique presentable $\infty$-categorical enhancements of $\widehat{\D}(A)$ in the spirit of
    this paper. The argument of~\cite{cns} for unique models for $\Perf(X)$ and
    $\D_\qc(X)$ is new and involves a local-global argument which would be nice
    to return to in our setting.
\end{mainremark}

In general, it can happen that a triangulated category admits multiple dg enhancements
but a unique stable $\infty$-categorical enhancement (we give an example due to Dugger and
Shipley~\cite{dugger-shipley-topological} in
Example~\ref{ex:dugger-shipley}). This does not occur in the situations above because the
presence of a $0$-complicial $t$-structure guarantees the existence of a canonical $\ZZ$-linear enrichment. 

\begin{mainmetatheorem}\label{meta:dg}
    In all of the cases above, the triangulated categories admit unique dg enhancements.
\end{mainmetatheorem}

In Section~\ref{sub:nstable} we conjecture the
existence of a theory of stable $n$-categories and exact functors for each
$1\leq n\leq\infty$ and we give a
conjecture on a stable $n$-categorical analogue of Theorem~\ref{mt:2}. 
The $n=1$ theory is that of triangulated categories and exact functors and the
$n=\infty$ theory is that of stable $\infty$-categories and exact functors.
A typical stable $n$-category is the $n$-homotopy category $h_{n-1}\Cscr$ where $\Cscr$ is a stable
$\infty$-category (where our previous notation $h\Cscr$ agrees with $h_0\Cscr$).
One problem is to define intrinsic to $n$-categories what a stable $n$-category
should be via a list of axioms similar to those for a triangulated category.

%

We postpone further discussion to Section~\ref{sec:examples}, where
we give historical remarks, examples, questions, and propose several more conjectures.
Between now and then, we give background on stable and prestable
$\infty$-categories in Sections~\ref{sec:stable} and~\ref{sec:prestable}.
Section~\ref{sec:ba} gives a uniqueness statement for $\infty$-categorical
enhancements of bounded above derived categories which we will use to start our
arguments. In Section~\ref{sec:detection} we give a detection lemma saying
certain properties of $\infty$-categorical enhancements can be detected on the homotopy category.
Section~\ref{sec:proofs} contains the proofs of Theorems~\ref{mt:1},~\ref{mt:2},
and~\ref{mt:3}. In Section~\ref{sec:meta} we say something about Meta
Theorem~\ref{meta:dg} and we end with Appendix~\ref{sub:appendix} which removes
presentability from the statement of Theorem~\ref{mt:2}.

\paragraph{Conventions.} Throughout, we use homological indexing conventions. We also work
everywhere with $\infty$-categories. An $\infty$-category $\Cscr$ is canonically enriched in
spaces: if $X,Y\in\Cscr$, we write $\Map_\Cscr(X,Y)$ for the mapping space from $X$ to $Y$.
If $\Cscr$ is stable, then there is a canonical enrichment in spectra. We write
$\MapSp_\Cscr(X,Y)$ for the mapping spectrum. These are related by
$\Omega^\infty\MapSp_\Cscr(X,Y)\we\Map_\Cscr(X,Y)$. Thus,
$\pi_i\MapSp_\Cscr(X,Y)\iso\pi_i\Map_\Cscr(X,Y)$ for $i\geq 0$. We write
$h\Cscr$ for the homotopy category of an $\infty$-category. We have
$\pi_0\Map_\Cscr(X,Y)\iso\Hom_{h\Cscr}(X,Y)$. We will typically use script
letters such as $\Cscr$, $\Dscr_{\qc}(X)$, or $\Perfscr(X)$ to denote
$\infty$-categories and roman letters such as $\T$ or $\Perf(X)$ to denote
triangulated categories. One exception is that we typically write $\Ascr$ for
an abelian category.

\paragraph{Acknowledgments.}
This paper was conceived during a talk by Alice Rizzardo on uniqueness of
enhancements at Neeman's birthday conference in Bielefeld
in the summer of 2017. Rizzardo spoke about joint work~\cite{rizzardo-vdb} with
Michel van den Bergh
producing an example of a $k$-linear triangulated category for some field $k$
with multiple non-equivalent dg enhancements. This sparked our interest in the question in
general.

We would like to thank the following people for their helpful comments
and conversations: Jack Hall, Amnon Neeman, George Raptis, Nick Rozenblyum, Stefan Schwede, Jay
Shah, Greg Stevenson, and Wenliang Zhang. We thank Marc Hoyois, Irakli
Patchkoria, Greg Stevenson,
and Marco Varisco for providing additional bibliographic information and also
for pointing out some typos. Akhil Mathew found a mistake in the original proof
of Proposition~\ref{prop:complete} and Niko Naumann and Irakli Patchkoria pointed out that
Conjecture~\ref{conj:non} is true; we thank them both. Finally, and most importantly, we are indebted to Alberto Canonaco
and Paolo Stellari for pointing out a gap in our original ``proof''
of Theorem~\ref{mt:3} and for extensive feedback on the new proof. This work was supported by NSF Grant DMS-1552766.

\section{$\infty$-categorical enhancements}\label{sec:stable}

A stable $\infty$-category is a pointed $\infty$-category $\Cscr$ with finite
limits and finite colimits and such that a commutative square
$$\xymatrix{W\ar[r]\ar[d]&X\ar[d]\\Y\ar[r]&Z}$$ is a pushout if and only if it
is a pullback.

If $\Cscr$ is stable, then the homotopy category $h\Cscr$ is canonically
triangulated (\cite[1.1.2.14]{ha}). Note however that being stable is a {\em property} of an
$\infty$-category as opposed to extra structure. We will use stable
$\infty$-categories in this paper as the natural models of triangulated
categories. Other possible models are stable simplicial model categories and
dg categories; both are captured by using stable
$\infty$-categories (see~\cite[Appendix~A.2]{htt} for simplicial model
categories and~\cite[Section~1.3.1]{ha} for dg categories). 
We will assume familiarity with Lurie's work on higher categories,
especially~\cite{htt,ha,sag}.

\begin{definition} Let $\T$ be a triangulated category.
    We say that $\T$ {\bf admits an $\infty$-categorical
    enhancement} if there is a stable $\infty$-category
    $\Cscr$ and a triangulated equivalence $h\Cscr\we\T$. If
    $\Cscr$ is unique up to equivalence of $\infty$-categories, we say
    that $\T$ {\bf admits a unique $\infty$-categorical enhancement}.
\end{definition}

\begin{variant}
    We say that $\T$ {\bf admits a presentable $\infty$-categorical
    enhancement} if there is a stable presentable $\infty$-category $\Cscr$ and
    a triangulated equivalence $h\Cscr\we\T$. If
    $\Cscr$ is unique up to equivalence of $\infty$-categories, in the sense
    that if $\Dscr$ is a second stable {\em presentable} $\infty$-category such
    that $h\Dscr\we\T$, then $\Cscr\we\Dscr$, then we say
    that $\T$ {\bf admits a unique presentable $\infty$-categorical enhancement}.
\end{variant}

Basically all triangulated categories with small coproducts that appear in algebra,
homotopy theory, and algebraic geometry admit presentable $\infty$-categorical
models. This paper is about uniqueness.


\begin{definition}
    Let $\T$ be a triangulated category which admits small coproducts.
    A set of objects $\{X_i\}$ in $\T$ {\bf generates} $\T$ if the following
    condition holds: if $Y\in\T$ satisfies $\Hom_\T(X_i,Y[n])=0$ for all $X_i$
    and $n\in\ZZ$, then $Y\we 0$.
\end{definition}

\begin{definition}
    Let $\T$ be a triangulated category with all small coproducts. An object
    $X\in\T$ is {\bf compact} (or
    $\omega$-compact) if for all coproducts $\coprod_{i\in I}Y_i$ the natural
    map $$\coprod_{i\in I}\Hom_\T(X,Y_i)\rightarrow\Hom_\T(X,\coprod_{i\in
    I}Y_i)$$ is a bijection. We let $\T^\omega\subseteq\T$ be the full
    subcategory of compact objects, which inherits a triangulated structure
    from $\T$. A triangulated category $\T$ is {\bf compactly
    generated} if it is
    locally small, has all small coproducts, and is generated by $\T^\omega$.
\end{definition}

\begin{proposition}[Lurie]\label{prop:presentable}
    Suppose that $\T$ is compactly generated and admits an $\infty$-categorical
    model $\Cscr$. Then, $\Cscr$ is presentable.
\end{proposition}

\begin{proof}
    See~\cite[1.4.4.2 and 1.4.4.3]{ha}.
\end{proof}

\begin{warning}
    Neeman has a notion of well generated $\infty$-category and it would be
    good to know that if $\T$ is well generated and admits an
    $\infty$-categorical model $\Cscr$, then $\Cscr$ is presentable. However,
    the key implication in~\cite[1.4.4.2]{ha} is specific to the compact case.
    We are not sure whether or not this is true and we will have to take care
    to figure out what is happening in our cases of interest.
\end{warning}

Now, we recall some facts about $t$-structures.

\begin{definition}\label{def:ttriangulated}
    Let $\T$ be a triangulated category. A $t$-structure on $\T$ is a pair of full
    subcategories $(\T_{\geq 0},\T_{\leq 0})$ such that
    \begin{enumerate}
        \item[(i)]  $\T_{\geq 0}[1]\subseteq\T_{\geq 0}$, $\T_{\leq 0}[-1]\subseteq\T_{\leq
            0}$;
        \item[(ii)] if $X\in\T_{\geq 0}$ and $Y\in\T_{\leq 0}$, then $\Hom_\T(X,Y[-1])=0$;
        \item[(iii)] every object $X$ fits into an exact triangle $\tau_{\geq 0}X\rightarrow
            X\rightarrow \tau_{\leq -1}X$ where $\tau_{\geq 0}X\in\T_{\geq 0}$ and $\tau_{\leq
            -1}X[1]\in\T_{\leq 0}$.
    \end{enumerate}
\end{definition}

There is an entirely parallel notion for stable $\infty$-categories.

\begin{definition}\label{def:tstable}
    Let $\Cscr$ be a stable $\infty$-category. A $t$-structure on $\Cscr$ is a pair of full
    subcategories $(\Cscr_{\geq 0},\Cscr_{\leq 0})$ such that
    \begin{enumerate}
        \item[(i)]  $\Cscr_{\geq 0}[1]\subseteq\Cscr_{\geq 0}$, $\Cscr_{\leq
            0}[-1]\subseteq\Cscr_{\leq
            0}$;
        \item[(ii)] if $X\in\Cscr_{\geq 0}$ and $Y\in\Cscr_{\leq 0}$, then $\Map_\Cscr(X,Y[-1])=0$;
        \item[(iii)] every object $X$ fits into an exact triangle $\tau_{\geq 0}X\rightarrow
            X\rightarrow \tau_{\leq -1}X$ where $\tau_{\geq 0}X\in\Cscr_{\geq 0}$ and $\tau_{\leq
            -1}X[1]\in\Cscr_{\leq 0}$.
    \end{enumerate}
\end{definition}

\begin{remark}\label{rem:2prime}
    \begin{enumerate}
        \item[(a)] Condition (ii) of Definition~\ref{def:tstable} is equivalent to
            \begin{enumerate}
                \item[(ii')] if $X\in\Cscr_{\geq 0}$ and $Y\in\Cscr_{\leq 0}$, then
                $\Hom_{h\Cscr}(X,Y[-1])=0$.
            \end{enumerate}
        \item[(b)] The truncations $\tau_{\geq n}X$ and $\tau_{\leq n}X$ are functorial:
            $\tau_{\geq n}$ is the right adjoint to the inclusion of $\Cscr_{\geq n}$ in
            $\Cscr$, and $\tau_{\leq n}$ is the left adjoint to the inclusion of $\Cscr_{\leq
            n}$ in $\Cscr$. The $n$th homotopy object $\pi_nX$ of $X$ is an object of the abelian
            category $\Cscr^\heart=\Cscr_{\geq 0}\cap\Cscr_{\leq 0}$, defined as $\tau_{\geq n}\tau_{\leq n}X[-n]$. Given a fiber
            sequence $X\rightarrow Y\rightarrow Z$, one obtains a natural long exact sequences
            $$\cdots\rightarrow\pi_n
            X\rightarrow\pi_nY\rightarrow\pi_nZ\rightarrow\pi_{n-1}X\rightarrow\cdots$$
            in $\Cscr^\heart$.
    \end{enumerate}
\end{remark}

\begin{lemma}\label{lem:t}
    Let $\Cscr$ be a stable $\infty$-category.
    The data of a $t$-structure on $\Cscr$ is equivalent to the data of a $t$-structure on
    the triangulated category $h\Cscr$.
\end{lemma}

\begin{proof}
    Given a $t$-structure $(\Cscr_{\geq 0},\Cscr_{\leq 0})$, then the pair
    $h\Cscr_{\geq 0}\subseteq h\Cscr$ and $h\Cscr_{\leq 0}\subseteq h\Cscr$ of full
    subcategories defines a $t$-structure on $h\Cscr$ (see also Remark~\ref{rem:2prime}).
    Let $h\colon\Cscr\rightarrow h\Cscr$ be the natural functor.
    Similarly, given a $t$-structure $(\T_{\geq 0},\T_{\leq 0})$ on $h\Cscr$, let
    $\Cscr_{\geq 0}$ be the full subcategory of those objects $X\in\Cscr$ such that
    the image of $X$ in the homotopy category is in the subcategory $\T_{\geq 0}$
    and similarly for $\Cscr_{\leq 0}$. It is easy to check that
    these define a $t$-structure on $\Cscr$.
\end{proof}

The point of Lemma~\ref{lem:t} for us will be that $t$-structures go along for the
ride when considering enhancements.

We are now interested in a flurry of special properties of $t$-structures.

\begin{definition}\label{def:tomnibus}
Let $(\T_{\geq 0},\T_{\leq 0})$ be a $t$-structure on a triangulated
category $\T$ and let $(\Cscr_{\geq 0},\Cscr_{\leq 0})$ be a $t$-structure on a stable
$\infty$-category $\Cscr$. Set $\T_{\geq n}=\T_{\geq 0}[n]$, $\T_{\leq
n}=\T_{\leq 0}[n]$, $\Cscr_{\geq n}=\Cscr_{\geq 0}[n]$, and $\Cscr_{\leq n}=\Cscr_{\leq 0}[n]$.
    \begin{enumerate}
        \item[(a)]  The $t$-structure on $\T$ is {\bf left separated} if $$\cap_{n\in\ZZ}\T_{\geq
            n}=0.$$ It is {\bf right separated} if $$\cap_{n\in\ZZ}\T_{\leq n}=0.$$
        \item[(b)] The $t$-structure on $\Cscr$ is {\bf left separated} if
            $$\cap_{n\in\ZZ}\Cscr_{\geq
            n}=0.$$ It is {\bf right separated} if $$\cap_{n\in\ZZ}\Cscr_{\leq n}=0.$$
        \item[(c)] The $t$-structure on $\Cscr$ is {\bf left complete} if the natural map
            \begin{equation}\label{eq:leftcomplete}\Cscr\rightarrow\lim\left(\cdots\rightarrow\Cscr_{\leq 2}\xrightarrow{\tau_{\leq
                    1}}\Cscr_{\leq 1}\xrightarrow{\tau_{\leq 0}}\Cscr_{\leq
                    0}\right)\end{equation} is an equivalence. It is {\bf right
                complete} if the natural map
                $$\Cscr\rightarrow\lim\left(\cdots\rightarrow\Cscr_{\geq
                        -2}\xrightarrow{\tau_{\geq -1}}\Cscr_{\geq
                        -1}\xrightarrow{\tau_{\geq 0}}\Cscr_{\geq 0}\right)$$ is an
                equivalence.
        \item[(d)] Suppose that $\Cscr$ is presentable. We say that the $t$-structure is
            {\bf accessible} if $\Cscr_{\geq 0}$ is presentable. This happens if and only if
            $\Cscr_{\leq 0}$ is presentable. See~\cite[1.4.4.13]{ha}.
        \item[(e)] Suppose that $\Cscr$ has filtered colimits. We say that the $t$-structure is
            {\bf compatible with filtered colimits} if $\Cscr_{\leq 0}$ is
            closed under filtered colimits
            in $\Cscr$.
        \item[(f)] Suppose that $\Cscr$ has countable products. We say that $\Cscr$ is {\bf compatible with countable
            products} if $\Cscr_{\geq 0}$ is closed under countable products in $\Cscr$.
        \item[(g)] Suppose that $\T$ admits all small coproducts. We say that
            the $t$-structure $(\T_{\geq 0},\T_{\leq 0})$ is {\bf compatible
            with filtered homotopy colimits} if $\T_{\leq 0}$ is closed under
            filtered homotopy colimits in $\T$.
    \end{enumerate}
\end{definition}

\begin{warning}
    It is common to say that a $t$-structure is separated if it is both left and right
    separated. We will {\em never} do this. Instead, the notion of being separated (and
    complete) is reserved for prestable $\infty$-categories and will be introduced
    in Section~\ref{sec:prestable}.
\end{warning}

\begin{example}
    \begin{enumerate}
        \item[(i)]  The $\infty$-category $\Sp$ of spectra with its Postnikov $t$-structure
            is left and right complete, accessible, and compatible with filtered colimits.
            See~\cite[1.4.3.6]{ha}.
        \item[(ii)] The derived $\infty$-category $\Dscr(A)$ of any associative
            ring (or connective $\EE_1$-ring spectrum) together with
            its Postnikov (or standard) $t$-structure is left and right complete, accessible, and compatible
            with filtered colimits. See~\cite[7.1.1.13]{ha}.
        \item[(iii)] If $\Ascr$ is a Grothendieck abelian category, then the derived
            $\infty$-category $\Dscr(\Ascr)$ is left and right separated, right complete,
            accessible, and compatible with filtered colimits (see~\cite[1.3.5.21]{ha}). It is not typically left
            complete: see Example~\ref{ex:neeman}, due to Neeman, below.
        \item[(iv)] If $X$ is a quasi-compact scheme with affine diagonal, then $\Dscr(\QCoh(X))$,
            the derived $\infty$-category of quasi-coherent sheaves on $X$ is left and right
            complete, accessible, and compatible with filtered colimits. Indeed, everything
            except for left completeness follows from point (iii). But, in this case,
            $\Dscr(\QCoh(X))\we\Dscr_\qc(X)$ (the B\"okstedt--Neeman
            proof~\cite{bokstedt-neeman} in the quasi-compact and separated case immediately applies
            to the case of a quasi-compact scheme with affine diagonal) and $\Dscr_\qc(X)$ is always left complete as
            it is a limit of left complete $t$-structure along $t$-exact functors.
        \item[(v)] Consider $\Dscr(\ZZ)$ and fix a prime number $p$. The kernel of the
            localization $\Dscr(\ZZ)\rightarrow\Dscr(\ZZ[\tfrac{1}{p}])$ will be written
            $\Dscr(\Tors_p)$; it is the derived category of the Grothendieck abelian category of
            $p$-primary torsion abelian groups. With the Postnikov $t$-structure (induced
            from $\Dscr(\ZZ)$), it is left and right complete, accessible, and
            compatible with filtered colimits. However, there is an equivalence
            $$\Dscr(\Tors_p)\we\Dscr(\ZZ)_{\widehat{p}}$$
            where the latter is the $\infty$-category of derived $p$-complete complexes of
            abelian groups. There is a different $t$-structure on derived $p$-complete
            abelian groups, induced from the fully faithful inclusion into $\Dscr(\ZZ)$.
            It is left and right separated and even right complete by
            Proposition~\ref{prop:sepcomplete} below. It is accessible, but not compatible
            with filtered colimits. Indeed, $\colim_n\ZZ/p^n$ is the completion of
            $\QQ_p/\ZZ_p$ which is $\ZZ_p[1]$, so we see that the coconnective objects are
            not closed under filtered colimits in $\Dscr(\ZZ)_{\widehat{p}}$.
    \end{enumerate}
\end{example}

\begin{example}
    If $(\T_{\geq 0},\T_{\leq 0})$ is a left separated $t$-structure on $\T$ and if $X\in\T$
    is such that $\tau_{\leq n}X\we 0$ for all $n$, then $X\we 0$. Indeed, in this case,
    $\tau_{\geq n+1}X\we X$ for all $n$, so that $X\in\cap_{n\in\ZZ}T_{\geq n}=0$.
\end{example}

\begin{lemma}\label{lem:sepcomplete}
    If $\Cscr$ is left or right complete, then it is left or right separated, respectively.
\end{lemma}

\begin{proof}
    Assume that $\Cscr$ is left complete. Let $X\in\cap_{n\in\ZZ}\Cscr_{\geq n}$. Then,
    $\tau_{\leq n}X\we 0$ for all $n$. Thus, $X$ is zero in the
    limit~\eqref{eq:leftcomplete}. Hence, $\Cscr$ is left separated. The proof in the right
    separated case is the same.
\end{proof}

There is an important partial converse due to Lurie.

\begin{proposition}\label{prop:sepcomplete}
    Let $\Cscr$ be a stable $\infty$-category with a $t$-structure $(\Cscr_{\geq
    0},\Cscr_{\leq 0})$. 
    \begin{enumerate}
        \item[{\rm (1)}] Suppose that $\Cscr$ admits countable coproducts and that $\Cscr_{\leq 0}$ is
            closed under countable coproducts in $\Cscr$. If the $t$-structure is right
            separated, then it is right complete.
        \item[{\rm (2)}] Suppose that $\Cscr$ admits countable products and
            that the $t$-structure is compatible with countable products. If the $t$-structure is left
            separated, then it is left complete.
    \end{enumerate}
\end{proposition}

\begin{proof}
    Part (2) is~\cite[1.2.1.19]{ha}. Part (1) follows from (2) by taking opposite
    categories and using that if $(\Cscr_{\geq 0},\Cscr_{\leq 0})$ is a $t$-structure on
    $\Cscr$, then $(\Cscr_{\leq 0}^\op,\Cscr_{\geq 0}^\op)$ is a $t$-structure on
    $\Cscr^\op$.
\end{proof}

\begin{warning}\label{warning:sep}
    It is tempting to guess that if $\Cscr$ is left separated, then the natural
    map $X\rightarrow\lim_n\tau_{\leq n}X$ is an equivalence for all $X$. This
    is certainly the case if $X$ is left complete. However, in general this is
    false. Suppose that $\Cscr_{\geq 0}$ is not closed under countable products
    in $\Cscr$. Let $\{X_i\}$ for $i\geq 0$ be a sequence of objects of
    $\Cscr_{\geq 0}$ such that $X_i\in\Cscr_{\leq i}$ and $\prod_{i=0}^\infty X_i$ is not in $\Cscr_{\geq 0}$.
    Consider the natural map $$\bigoplus_{i=0}^\infty X_i\rightarrow\lim_n\tau_{\leq
    n}\bigoplus_{i=0}^\infty X_i\we\lim_n\bigoplus_{i=0}^n X_i\we\prod_{i=0}^\infty X_i.$$
    The direct sum is evidently in $\Cscr_{\geq 0}$ but the product is not in
    $\Cscr_{\geq 0}$, by hypothesis, so the map is not an equivalence.
    For additional discussion, see Section~\ref{sub:completeness}.
\end{warning}

\begin{example}\label{ex:neeman}
    Neeman has shown in~\cite{neeman-left} that examples of this sort abound. In particular,
    $\Dscr(\Ascr)$ where $\Ascr$ is the Grothendieck abelian
    category of representations of $\Ga$ over a characteristic $p$ field is not
    left complete.
\end{example}

%

We will need to have a general condition for accessibility of a $t$-structure.

\begin{lemma}\label{lem:access}
    Let $\Cscr$ be a stable presentable $\infty$-category with a $t$-structure $(\Cscr_{\geq
    0},\Cscr_{\leq 0})$. Suppose that there is a set of objects $\{X_i\}_{i\in I}$ of
    $\Cscr_{\geq 0}$ such that $\Cscr_{\geq 0}$ is the smallest subcategory of $\Cscr$
    containing the $\{X_i\}_{i\in I}$ and closed under colimits and extensions in $\Cscr$.
    Then, $\Cscr_{\geq 0}$ is presentable.
\end{lemma}

\begin{proof}
    See~\cite[1.4.4.11]{sag}.
\end{proof}

\section{Prestable $\infty$-categories}\label{sec:prestable}

Let $\T$ be a triangulated category with a $t$-structure $(\T_{\geq 0},\T_{\leq
0})$. A prestable $\infty$-category is to $\T_{\geq 0}$ as a stable
$\infty$-category is to $\T$. Such objects have not been studied in the world
of dg categories, but the homotopy categories have received some small amount
of attention in~\cite{keller-vossieck-sous,keller-vossieck-aisles} under the
name of {\em suspended categories} or {\em aisles}. Most work, as
in~\cite{tarrio-etal,hrbek}, has focused
on the classification of aisles inside a fixed triangulated category,
rather than the categorical properties of the aisles themselves.

The primary feature of prestable $\infty$-categories is that the residue of the
$t$-structure is not extra structure but rather an inherent feature. In
particular, every prestable $\infty$-category $\Dscr$ has a heart
$\Dscr^\heart$, which is equivalent to the nerve of an additive category sitting fully
faithfully inside $\Dscr$. In many cases of interest, such as when $\Dscr$ is
the connective part of some $t$-structure, $\Dscr^\heart$ is abelian.
The point for us is that often there are {\em unique} prestable $\infty$-categories having certain
properties and with a certain heart.

The definitions below are due to Lurie~\cite{sag}.

\begin{definition}\label{def:prestable}
    An $\infty$-category $\Cscr$ is {\bf prestable} if
    \begin{enumerate}
        \item[(a)]  $\Cscr$ is pointed and has finite colimits,
        \item[(b)]  the suspension functor $\Sigma=[1]\colon\Cscr\rightarrow\Cscr$ is fully
            faithful;
        \item[(c)]  if $u\colon X\rightarrow Y[1]$ is a map in $\Cscr$, then $u$ admits a fiber.
        \end{enumerate}
\end{definition}

\begin{remark}
    \begin{enumerate}
        \item[(1)]    Lurie shows that $\Cscr$ is prestable if and only if it admits a fully faithful functor
            $\Cscr\rightarrow\Dscr$ such that $\Dscr$ is stable and the essential image of $\Cscr$
            is closed under finite colimits and extensions in $\Dscr$. In fact, we can take
            $\Dscr\we\SWscr(\Cscr)=\colim\left(\Cscr_{\geq 0}\xrightarrow{\Sigma}\Cscr_{\geq
            0}\xrightarrow{\Sigma}\Cscr_{\geq 0}\rightarrow\cdots\right)$.
        \item[(2)]    Let $\Cscr$ be a prestable $\infty$-category and let $\Cscr_{\leq n}\subseteq\Cscr$ be
            the full subcategory of $n$-truncated objects. Then,
            $\Cscr^\heart=\Cscr_{\leq 0}$ is equivalent to (the nerve of) an additive category.
        \item[(3)]  A prestable $\infty$-category has finite limits if and only if it is the
            connective part of a $t$-structure on some stable $\infty$-category
            $\Dscr$. In this case,
            $\Cscr^\heart=\Dscr^\heart$ is an abelian category by~\cite[1.3.6]{bbd}.
            Again, we can take $\Dscr\we\SWscr(\Cscr)$.
        \item[(4)] When $\Cscr$ is a prestable $\infty$-category with finite limits, we can
            construct the $\infty$-category
            $$\Sp(\Cscr)=\lim\left(\cdots\rightarrow\Cscr\xrightarrow{\Omega}\Cscr\xrightarrow{\Omega}\Cscr\right)$$
            of spectrum objects in $\Cscr$. In this case, there is a fully faithful inclusion
            $\Cscr\rightarrow\Sp(\Cscr)$ which is the connective part of a $t$-structure on
            $\Sp(\Cscr)$.
   \end{enumerate}
\end{remark}

\begin{example}\label{ex:prestable}
    \begin{enumerate}
        \item[(a)]   If $\Cscr$ is a stable $\infty$-category with a $t$-structure $(\Cscr_{\geq
                0},\Cscr_{\leq 0})$, then $\Cscr_{\geq 0}$ is a prestable $\infty$-category.
                We will mostly study a special case, namely $\Dscr(\Ascr)_{\geq 0}$ when $\Ascr$ is
                a Grothendieck abelian category.
        \item[(b)]   If $R$ is a commutative ring, then $\Perfscr(R)_{\geq
            0}=\Perfscr(R)\cap\Dscr(R)_{\geq 0}$ is closed under extensions and finite
            colimits in $\Perfscr(R)$. It follows that $\Perfscr(R)_{\geq 0}$ is prestable.
            It is typically not the connective part of a $t$-structure on $\Perfscr(R)$. In
            fact, this holds if and only if $R$ satisfies some strong regularity conditions.
        \item[(c)] The $\infty$-category $\Sp_{\geq 0}^\omega$ of compact connective spectra
            is a prestable $\infty$-category. Again, this is not the connective part of a
            $t$-structure on compact spectra $\Sp^\omega$.
        \item[(d)]  Let $\Ascr$ be a small additive category. We let
            $\Pscr_\Sigma(\Ascr)=\Fun^\pi(\Ascr^\op,\Sscr)$, the $\infty$-category of finite
            product preserving functors from $\Ascr^\op$ to the $\infty$-category of spaces.
            This is equivalent to $\Fun^\pi(\Ascr^\op,\Sp_{\geq 0})$ and also to
            $\Dscr(\Mod_{\Ascr})_{\geq 0}$, where
            $\Mod_{\Ascr}=\Fun^\pi(\Ascr^{\op},\Mod_\ZZ)$ is the
            Grothendieck abelian category of additive functors from $\Ascr^\op$ to the category
            of abelian groups. In this case, the $\infty$-category
            $\Sp(\Pscr_\Sigma(\Ascr))$ of spectrum objects in $\Pscr_\Sigma(\Ascr)$ is equivalent to
            $\Fun^\pi(\Ascr^\op,\Sp)\we\Dscr(\Mod_\Ascr)$.
    \end{enumerate}
\end{example}

\begin{definition}\label{def:prestableomnibus}
    Let $\Cscr$ be a prestable $\infty$-category which admits finite limits.
    \begin{itemize}
        \item[(i)]  We say that $\Cscr$ is {\bf separated} if for an object $X$
            the condition $\tau_{\leq n}X\we 0$ for all $n\geq 0$ implies $X\we 0$.
        \item[(ii)] We say that $\Cscr$ is {\bf complete} if the natural map
        $$\Cscr\rightarrow\lim\left(\cdots\rightarrow\Cscr_{\leq 2}\xrightarrow{\tau_{\leq
                1}}\Cscr_{\leq 1}\xrightarrow{\tau_{\leq 0}}\Cscr_{\leq 0}\right)$$ is an
        equivalence.
        \item[(iii)] We say that $\Cscr$ is {\bf Grothendieck prestable} if it is presentable and
            filtered colimits are left exact.
    \end{itemize}
\end{definition}

\begin{remark}
    \begin{enumerate}
        \item[(i)] If $\Cscr$ is a stable $\infty$-category with a
            $t$-structure $(\Cscr_{\geq 0},\Cscr_{\leq 0})$, then $\Cscr$ is
            left separated if and only if the prestable $\infty$-category $\Cscr_{\geq 0}$ is separated.
        \item[(ii)] Similarly, $\Cscr$ is left complete if and only if
            $\Cscr_{\geq 0}$ is complete. This follows for example
            from~\cite[1.2.1.17]{ha}.
        \item[(iii)]
            Finally, a prestable $\infty$-category $\Cscr$ is Grothendieck prestable if and only if there is a stable presentable
            $\infty$-category $\Dscr$ with an accessible $t$-structure $(\Dscr_{\geq 0},\Dscr_{\leq
            0})$ that is compatible with filtered colimits and such that $\Dscr_{\geq
            0}\we\Cscr$. See~\cite[C.1.4.1]{sag}.
    \end{enumerate}
\end{remark}

As in the case of stable $\infty$-categories, complete implies separated. There is another
crucial sequence of definitions, again all due to Lurie~\cite{sag}.

\begin{definition}
    Let $\Cscr$ be a Grothendieck prestable $\infty$-category.
    \begin{enumerate}
        \item[(a)]   Say that $\Cscr$ is {\bf $n$-complicial} if for every object $Y\in\Cscr$ there is an
        object $X\in\Cscr_{\leq n}$ and a map $X\rightarrow Y$ inducing a surjection
        $\pi_0X\rightarrow\pi_0Y$ in $\Cscr^\heart$.
        \item[(b)]   Say that $\Cscr$ is {\bf weakly $n$-complicial} if the above
            condition holds for every $Y$ such that $Y\in\Cscr_{\leq m}$ for
            some $m$ (i.e., it holds for the bounded above $Y$).
        \item[(c)]   Say that $\Cscr$ is {\bf anticomplete} if the natural map
        $\Fun^\L(\Cscr,\Dscr)\rightarrow\Fun^\L(\Cscr,\widehat{\Dscr})$ is an equivalence
        for every Grothendieck prestable $\infty$-category $\Dscr$, where
        $\widehat{\Dscr}=\lim_n\Dscr_{\leq n}$ is the completion of $\Dscr$ and where
        $\Fun^\L(-,-)$ denotes the $\infty$-category of colimit preserving functors.
    \end{enumerate}
\end{definition}

\begin{example}
    \begin{enumerate}
        \item[(i)] If $R$ is a dg algebra with $\H_i(R)=0$ for $i<0$, then
            $\Dscr(R)_{\geq 0}$ is $n$-complicial if and only if $\H_i(R)=0$ for $i>n$.
            See~\cite[C.5.5.15]{sag}.
        \item[(ii)] If $\Ascr$ is a Grothendieck abelian category, then $\Dscr(\Ascr)_{\geq
            0}$ is $0$-complicial. For the simple argument, see~\cite[C.5.3.2]{sag}.
        \item[(iii)] If $X$ is a quasi-compact scheme with affine diagonal, then
            $\Dscr_{\qc}(X)_{\geq 0}$ is $0$-complicial. Indeed, in this case,
            $\Dscr_\qc(X)_{\geq 0}\we\Dscr(\QCoh(X))_{\geq 0}$ by~\cite{bokstedt-neeman}, so
            we conclude by (ii).
    \end{enumerate}
\end{example}

\section{Bounded above enhancements}\label{sec:ba}

To begin, we rephrase a result of Lurie in the present context. We write
$\D^+(\Ascr)$ for the homologically bounded above derived category when $\Ascr$
has enough injectives and $\D^-(\Ascr)$ for the homologically bounded below
derived category when $\Ascr$ has enough projectives.

\begin{proposition}\label{prop:ba}
    Let $\Ascr$ be an abelian category with enough injectives. The bounded
    above derived category $\D^+(\Ascr)$ admits a unique $\infty$-categorical enhancement.
\end{proposition}

\begin{proof}
    This is basically the content of~\cite[1.3.3.7]{ha} (see
    also~\cite[1.3.2.8]{ha}). Let $\Dscr^+(\Ascr)$ be the bounded above derived
    $\infty$-category constructed as in~\cite[1.3.2.7]{ha} using the dg nerve.
    This is an $\infty$-categorical model for $\D^+(\Ascr)$, so an enhancement
    exists. Now, suppose that $\Cscr$ is a general enhancement. As in
    Lemma~\ref{lem:t}, the $t$-structure on $\D^+(\Ascr)$ lifts to a
    $t$-structure on $\Cscr$ with heart $\Cscr^\heart\we\Ascr$. It follows
    from~\cite[1.3.3.7]{ha} that there exists a unique (up to homotopy)
    $t$-exact functor
    $\Dscr^+(\Ascr)\rightarrow\Cscr$ inducing an equivalence on hearts.
    Moreover, if $X\in\Ascr$ is some object and $Y\in\Ascr$ is injective, then
    $$\Ext^i_\Cscr(X,Y)\iso\Ext^i_\Ascr(X,Y)=0$$ for $i>0$ since
    $h\Cscr\we\D^+(\Ascr)$. To be entirely precise, we have a diagram
    $$\xymatrix{
        \Dscr^+(\Ascr)\ar[r]\ar[d]  &   \Cscr\ar[d]\\
        \D^+(\Ascr)\ar[r]&h\Cscr,
    }$$ where the top map is induced by the universal property of
    $\Dscr^+(\Ascr)$, the vertical maps are the truncations,\footnote{The nerve construction
        gives a fully faithful functor from $1$-categories into $\infty$-categories. We
            therefore view any category as an $\infty$-category. Taking the homotopy
            category is the left adjoint of this inclusion, and $\Cscr\rightarrow h\Cscr$ is
            the unit map of the adjunction.} and the bottom map
    is the fixed equivalence from the hypothesis that $\Cscr$ is an enhancement
    of $\D^+(\Ascr)$. We are {\em not} asserting that this square is
    commutative. However, it is commutative when restricted to hearts. Since
    all of the functors involved commute with the translation functors $[1]$,
    it follows that we can compute $\Hom_{h\Cscr}(X,Y[n])$ as claimed.
    Therefore, $\Dscr^+(\Ascr)\rightarrow\Cscr$ is fully faithful
    by~\cite[1.3.3.7]{ha}. The
    essential image is $\Cscr^+\subseteq\Cscr$, the full subcategory of bounded
    above objects in the $t$-structure. However, every object of $\Cscr$ is
    bounded above as this may be checked on the homotopy category. This
    completes the proof.
\end{proof}

\begin{variant}
    There is an entirely similar fact about abelian categories with enough
    projectives and bounded below derived categories. Note that Corollary~\ref{maincor:5} removes the assumption of having enough
    injectives or projectives from these statements for small abelian categories.
\end{variant}

\begin{remark}
    Lurie's proof of the crucial fact~\cite[1.3.3.7]{ha} used in the proof of
    Proposition~\ref{prop:ba} is rather different from the approach used
    by Lunts--Orlov and Canonaco--Stellari and works by identifying a universal property
    for $\Dscr^-(\Ascr)_{\leq 0}$ (see~\cite[1.3.3.8]{ha}). In particular, by using it below in the proofs
    of Theorems~\ref{thm:1} and~\ref{thm:2}, we are not simply reformulating the proofs
    of~\cite{lunts-orlov,cs-uniqueness}.
\end{remark}

\section{A detection lemma}\label{sec:detection}

Several theorems below rely on the ability to detect certain properties of a
$t$-structure on the homotopy category. We compile these in the following,
basically trivial, lemma.

\begin{lemma}\label{lem:detection}
    Let $\Cscr$ be a stable $\infty$-category with a $t$-structure
    $(\Cscr_{\geq 0},\Cscr_{\leq 0})$.
    \begin{enumerate}
        \item[{\em (a)}]
            The $t$-structure $(\Cscr_{\geq 0},\Cscr_{\leq 0})$ is left or right separated if and
            only if the same is true of the $t$-structure $(h\Cscr_{\geq 0},h\Cscr_{\leq 0})$ on $h\Cscr$.
        \item[{\em (b)}] The $t$-structure is compatible with countable products if
            and only if $h\Cscr_{\geq 0}$ is closed under countable products in
            $h\Cscr$.
        \item[{\em (c)}] Suppose now that $\Cscr^\heart$ has enough injectives
            and that $\Dscr^+(\Cscr^\heart)\rightarrow\Cscr^+$ is an
            equivalence. The $t$-structure on $\Cscr$ is compatible with
            filtered colimits if and only if the same is true of
            $\Dscr^+(\Cscr^\heart)$.
    \end{enumerate}
\end{lemma}

\begin{proof}
    Point (a) is clear. For point (b), let $X$ be an object of $\Cscr$ and $\{Y_i\}$ be a
    collection of objects of $\Cscr$. Note that
    \begin{align*}
        \Hom_{h\Cscr}(X,\prod_iY_i)&\iso\pi_0\Map_\Cscr(X,\prod_i Y_i)\\
        &\iso\pi_0\prod_i\Map_\Cscr(X,Y_i)\\
        &\iso\prod_i\pi_0\Map_\Cscr(X,Y_i)\\
        &\iso\prod_i\Hom_{h\Cscr}(X,Y_i),
    \end{align*}
    which shows that $\Cscr\rightarrow h\Cscr$ preserves products. (The same
    argument shows that it preserves coproducts, which we will use below.) Now,
    given a product $\prod_iY_i$ of objects $Y_i$ in $\Cscr_{\geq 0}$, the
    product $\prod_i Y_i$ is in $\Cscr_{\geq 0}$ if and only if its image
    in $h\Cscr$ is in $h\Cscr_{\geq 0}$. This proves (b).

    To prove (c), it is enough to prove that in general a $t$-structure on
    $\Cscr$ is compatible with filtered colimits if and only if the
    $t$-structure on $\Cscr^+$ is. Suppose that the $t$-structure on $\Cscr^+$
    is compatible with filtered colimits. The inclusion
    $\Cscr^+\hookrightarrow\Cscr$ preserves all coproducts that exist in
    $\Cscr^+$. It follows that filtered colimits of bounded above objects are
    bounded above (since these may be computed as a cofiber of a map between
    coproducts of bounded above objects), so that $\Cscr$ is compatible with filtered colimits. The
    other direction is clear.
\end{proof}

\section{Proofs}\label{sec:proofs}

We will repeatedly use the next lemma. Recall that if $\Cscr$ is a pointed
$\infty$-category with finite limits, the $\infty$-category of spectrum objects $\Sp(\Cscr)$ is
given as the limit of the diagram
$\cdots\rightarrow\Cscr\xrightarrow{\Omega}\Cscr\xrightarrow{\Omega}\Cscr$.

\begin{lemma}\label{lem:sp}
    Let $\Cscr$ and $\Dscr$ be stable $\infty$-categories with
    right complete $t$-structures. If $\Cscr_{\geq 0}\we\Dscr_{\geq
    0}$, then $\Cscr\we\Dscr$.
\end{lemma}

\begin{proof}
    In this case, right completeness implies that $\Cscr\we\Sp(\Cscr_{\geq 0})$
    and similarly for $\Dscr$.
\end{proof}

We use Proposition~\ref{prop:ba} and Lemma~\ref{lem:sp} to prove Theorem~\ref{mt:2}. 

\begin{theorem}\label{thm:2}
    Let $\Ascr$ be Grothendieck abelian. Then, the triangulated category
    $$\D(\Ascr)\hspace{10pt}$$
    admits a unique presentable $\infty$-categorical enhancement.
\end{theorem}

\begin{proof}
    That $\D(\Ascr)$ admits a presentable $\infty$-categorical model
    $\Dscr(\Ascr)$ is~\cite[1.3.5.21]{ha}. Lurie proves that the $t$-structure is
    accessible, right complete, left separated, and compatible with filtered
    colimits.
    Let $\Cscr$ be a stable presentable enhancement of $\D(\Ascr)$.
    Then, $\Cscr$ admits a $t$-structure, which is right complete by
    Proposition~\ref{prop:sepcomplete} and Lemma~\ref{lem:detection}, and we find that the full subcategory
    $\Cscr^+\subseteq\Cscr$ of bounded above objects is equivalent to
    $$\Dscr^+(\Ascr)$$
    using Proposition~\ref{prop:ba}. In particular, $\Cscr_{\leq 0}$ is presentable. It
    follows from~\cite[1.4.4.13]{ha} that $\Cscr_{\geq 0}$ is presentable (this is where we
    use presentability of $\Cscr$). Moreover, by
    Lemma~\ref{lem:detection}, the $t$-structure on $\Cscr$ is compatible with filtered
    colimits since this is true for $\Dscr(\Ascr)$. It follows that $\Cscr_{\geq 0}$ is Grothendieck prestable. It is
    $0$-complicial since this may be checked on the homotopy category. Finally,
    it is also left separated by Lemma~\ref{lem:detection} again. 
    But, by~\cite[C.5.4.5]{sag}, $\Dscr(\Ascr)_{\geq 0}$ is the unique $0$-complicial separated
    Grothendieck prestable $\infty$-category with heart $\Ascr$. So, we have
    $\Dscr(\Ascr)_{\geq 0}\we\Cscr_{\geq 0}$ which finishes the proof by
    Lemma~\ref{lem:sp}.
\end{proof}

A weaker version of this theorem appears as~\cite[C.5.4.11]{sag}. We just check
that on any presentable enhancement of $\Cscr$, the induced
$t$-structure shares the nice $\infty$-categorical properties of the $t$-structure on
$\Dscr(\Ascr)$.

\begin{corollary}
    If $\Ascr$ is Grothendieck abelian and $\D(\Ascr)$ is compactly generated, then
    $\D(\Ascr)$ admits a unique $\infty$-categorical enhancement.
\end{corollary}

\begin{proof}
    In this case, any $\infty$-categorical enhancement is presentable by
    Proposition~\ref{prop:presentable}, so the statement follows from Theorem~\ref{thm:2}.
\end{proof}

Now, we prove Theorem~\ref{mt:1} and Corollary~\ref{maincor:1}

\begin{theorem}\label{thm:1}
    If $\Ascr$ is a locally coherent Grothendieck abelian category, then the unseparated derived category $\check{\D}(\Ascr)$ admits a unique
    $\infty$-categorical enhancement.
\end{theorem}

\begin{proof}
    Lurie proves that there is an $\infty$-categorical enhancement
    in~\cite[Section~C.5.8]{sag}, but see also~\cite{krause-noetherian,krause-deriving}.
    Suppose that $\Cscr$ is an enhancement of $\cD(\Ascr)$. Since $\cD(\Ascr)$
    is compactly generated, $\Cscr$
    is presentable by Proposition~\ref{prop:presentable}. Using
    Proposition~\ref{prop:ba} we see that $\Cscr^+\we\cDscr^+(\Ascr)$ and in
    particular $\Cscr_{\leq 0}$ is presentable by~\cite[1.4.4.13]{ha}. Thus, $\Cscr_{\geq 0}$ is
    presentable. By Lemma~\ref{lem:detection}, the $t$-structure is compatible
    with filtered colimits. Thus, $\Cscr_{\geq 0}$ is Grothendieck prestable.
    Lurie proved in~\cite[C.5.5.20]{sag} that there is a unique anticomplete $0$-complicial
    Grothendieck prestable $\infty$-category $\cDscr(\Ascr)_{\geq 0}$ with heart $\Ascr$. Thus, to finish the
    proof, by Lemma~\ref{lem:sp}, it is enough to prove that $\Cscr_{\geq 0}$ is
    anticomplete and $0$-complicial. That $\Cscr_{\geq 0}$ is $0$-complicial follows
    immediately since we can detect this on the homotopy category.
    For anticompleteness, we argue as follows. Since $\Ascr^\omega$ is abelian, there is a
    natural equivalence
    $\cDscr(\Ascr)^\omega\we\Dscr^b(\Ascr^\omega)$ by~\cite[C.6.7.3]{sag}. As
    $h\Cscr^\omega\we h\cDscr(\Ascr)^\omega$, it follows that $\Cscr^\omega$ admits a
    bounded $t$-structure and that $\Ind(\Cscr^\omega_{\geq 0})\we\Cscr_{\geq 0}$. Since
    $\Ind$-completions of bounded $t$-structures have anticomplete Grothendieck connective
    parts, by~\cite[C.5.5.5]{sag}, we see that $\Cscr_{\geq 0}$ is anticomplete, as desired.
\end{proof}

\begin{corollary}\label{cor:1}
    If $\Ascr$ is a small abelian category, then $\D^b(\Ascr)$ admits a unique
    $\infty$-categorical enhancement.
\end{corollary}

\begin{proof}
    To see that there is an $\infty$-categorical model, take $\cDscr(\Ind(\Ascr))^\omega$,
    which has homotopy category $\Dscr^b(\Ascr)$ by~\cite[Theorem~4.9]{krause-deriving}.
    Let $\Cscr$ be a stable $\infty$-categorical enhancement of $\D^b(\Ascr)$.
    Then, $\Cscr$ admits a bounded $t$-structure. Thus, $\Ind(\Cscr)$ is a
    compactly generated stable presentable $\infty$-category with a
    $t$-structure $(\Ind(\Cscr_{\geq 0}),\Ind(\Cscr_{\leq 0}))$
    (see~\cite[Proposition~2.13]{agh} or~\cite[C.2.4.3]{sag}). Necessarily,
    $\Ind(\Cscr_{\geq 0})$ is anticomplete
    by~\cite[C.5.5.5]{sag}. We claim that $\Ind(\Cscr_{\geq 0})$ is also
    $0$-complicial and that $\Ind(\Cscr_{\geq 0})^\heart\we\Ind(\Ascr)$, which
    is enough to show that $\Ind(\Cscr_{\geq 0})\we\cDscr(\Ind(\Ascr))_{\geq
    0}$ by~\cite[C.5.5.19]{sag} (recalling that $\cDscr(\Ind(\Ascr))_{\geq 0}$
    is shown in~\cite[C.5.8.8]{sag} to be the unique anticomplete $0$-complicial
    Grothendieck prestable $\infty$-category with heart $\Ind(\Ascr)$). Given
    the claim, Lemma~\ref{lem:sp} implies that
    $\Ind(\Cscr)\we\cDscr(\Ind(\Ascr))$ and hence that
    $\Cscr\we\Ind(\Dscr)^\omega\we\cDscr(\Ind(\Ascr))^\omega\we\Dscr^b(\Ascr)$.
    Thus, let $X$ in $\Ind(\Cscr_{\geq 0})$ be an object. As $\Ind(\Cscr_{\geq 0})$ is compactly generated,
    there is a set of objects $\{Y_i\}$ of $\Cscr_{\geq 0}$ and morphism
    $\oplus Y_i\rightarrow X$ inducing a surjection on $\pi_0$. For each $Y_i$
    we can choose $Z_i\in\Cscr^\heart\we\Ascr$ and a map $Z_i\rightarrow Y_i$ inducing
    a surjection on $\pi_0$ (this can be checked in $\D^b(\Ascr)$ where it is
    clear by using brutal truncations). Thus, take the composition $\oplus
    Z_i\rightarrow\oplus Y_i\rightarrow X$. This proves that $\Ind(\Cscr_{\geq
    0})$ is $0$-complicial. The proof
    of~\cite[Proposition~2.13]{agh} proves that the $t$-structure on
    $\Ind(\Cscr)$ has connective part $\Ind(\Cscr_{\geq 0})$ and coconnective
    part $\Ind(\Cscr_{\leq 0})$. It follows that the heart of $\Ind(\Cscr_{\geq
    0})$ is $\Ind(\Cscr^\heart)\we\Ind(\Ascr)$, as desired. Here is another argument.
    We have a fully faithful colimit preserving exact functor
    $F\colon\Ind(\Ascr)\rightarrow\Ind(\Cscr_{\geq 0})^\heart$ and
    moreover every object of $\Ind(\Cscr_{\geq 0})^\heart$ receives a surjective map from
    an object in the essential image by the $0$-compliciality argument above.
    Let $U$ denote the right adjoint to $F$ (which exists by the adjoint
    functor theorem) and let $Y\in\Ind(\Cscr_{\geq 0})^\heart$. It is enough to prove
    that $UFY\rightarrow Y$ is an isomorphism. The fact that there exists a
    surjection $FX\rightarrow Y$ for some $X$ in $\Ind(\Ascr)$ implies that
    $UFY\rightarrow Y$ is surjective. Let $K$ be the kernel, so we have an exact sequence
    $$0\rightarrow K\rightarrow FUY\rightarrow Y\rightarrow 0$$ in
    $\Ind(\Cscr_{\geq 0})^\heart$. Applying $U$ and using that it is left exact, we get
    an exact sequence $$0\rightarrow UK\rightarrow UFUY\rightarrow UY.$$
    Since $UFUY\iso UY$, we see that $UK\we 0$. Let $FZ\rightarrow K$ be a
    surjection. It factors through maps $FZ\rightarrow FUK\rightarrow K$. Since
    $UK=0$, we see that the surjection factors through $0$ so that $K\we 0$.
    This is what we wanted to show.
\end{proof}

Finally, we prove Theorem~\ref{mt:3} and its corollaries.
Let $\Cscr$ be a prestable $\infty$-category.
Recall that we say that
$\Cscr$ is {\bf $0$-complicial} if
every for every object $X\in\Cscr$ there is an object
$Y\in\Cscr^\heart$ and a map $Y\xrightarrow{u} X$ such that the cofiber of $u$, computed in
$\SWscr(\Cscr)$, is in $\Cscr_{\geq 1}\we\Cscr[1]\subseteq\SWscr(\Cscr)$.

\begin{theorem}\label{thm:3}
    Let $\Cscr$ be a small idempotent complete prestable $\infty$-category. If $\Cscr$ is
    $0$-complicial, then the triangulated category $h\SWscr(\Cscr)$ admits a unique $\infty$-categorical
    enhancement.
\end{theorem}

\begin{proof}
    Let $\Escr$ be a stable $\infty$-category with an equivalence $h\Escr\we
    h\SWscr(\Cscr)$ and let $\Dscr\subseteq\Escr$ be the full subcategory of
    objects which correspond to the objects of $\Cscr$ under the equivalence.
    Then, $\Escr\we\SWscr(\Dscr)$. The equivalence $h\Escr\we h\SWscr(\Cscr)$ induces an equivalence
    $F\colon h\Cscr\we h\Dscr$. It will be enough to prove that
    $\Ind(\Cscr)\we\Ind(\Dscr)$. In that case,
    $\Cscr\we\Ind(\Cscr)^\omega\we\Ind(\Dscr)^\omega\we\Dscr$ and hence
    $\SWscr(\Cscr)\we\SWscr(\Dscr)\we\Escr$.

    Let $\Cscr^\heart$ be the full subcategory of
    $0$-truncated objects of $\Cscr$ and similarly for $\Dscr$. We
    evidently have an equivalence $F^\heart\colon\Cscr^\heart\we\Dscr^\heart$
    induced by $F$. We claim that $\Cscr^\heart$ is a set of generators for
    $\Ind(\Cscr)$, which also implies that $\Ind(\Cscr)$ is $0$-complicial.
    Fix $Z\in\Ind(\Cscr)$. We have to prove that there is a set $\{X_i\}$ of
    objects of $\Cscr^\heart$ together with a map $\bigoplus X_i\rightarrow Z$
    which induces a surjection on $\pi_0$ in $\Ind(\Cscr)^\heart$.
    Since $\Ind(\Cscr)$ is the ind-completion of a small prestable
    $\infty$-category, there is a map $\bigoplus Y_i\rightarrow Z$ inducing a
    surjection on $\pi_0$ for some
    collection of objects $\{Y_i\}\subseteq\Cscr$. Now, since $\Cscr$ is
    $0$-complicial, for each $Y_i$ there is a map $X_i\rightarrow Y_i$ which is
    a surjection on $\pi_0$ and where $X_i\in\Cscr^\heart$. The composition
    $\bigoplus X_i\rightarrow\bigoplus Y_i\rightarrow Z$ is the desired map.
    The same argument works to show that $\Dscr^\heart$ forms a set of
    generators for $\Ind(\Dscr)$ and that $\Dscr$ is $0$-complicial.

    Since $\Cscr^\heart\we\Dscr^\heart$,
    the $\infty$-categorical Gabriel--Popescu theorem~\cite[C.2.1.6]{sag} implies
    that both $\Ind(\Cscr)$ and $\Ind(\Dscr)$ are left exact localizations of
    $\Pscr_\Sigma(\Cscr^\heart)\we\Fun^{\pi}(\Cscr^{\heart,\op},\Sscr)$,
    the $\infty$-category of finite product preserving functors
    from $\Cscr^{\heart,\op}$ to the $\infty$-category of spaces (see also
            Example~\ref{ex:prestable}(d)). Let
    $$L_\Cscr\colon\Pscr_\Sigma(\Cscr^\heart)\rightleftarrows\Ind(\Cscr)\colon
    U$$ and $$L_\Dscr\colon\Pscr_\Sigma(\Cscr^\heart)\rightleftarrows\Ind(\Dscr)\colon
    V$$ by the two adjunctions.
    Let $\Kscr_\Cscr$ be the kernel of $L_\Cscr$.
    Let $u\colon W\rightarrow Z$ be a morphism in $\Pscr_\Sigma(\Cscr^\heart)$.
    Then, $L_\Cscr(u)$ is an equivalence if and only if $L_\Cscr(\cofib(u))\we
    0$ (see~\cite[C.2.3.2]{sag}). Let $S_\Cscr$ be the class of morphisms $u$ of
    $\Pscr_\Sigma(\Cscr^\heart)$ such that $L_\Cscr(u)$ is an equivalence.
    Define $\Kscr_\Dscr$ and $S_\Dscr$ similarly. We will be done if we show
    that $S_\Cscr=S_\Dscr$ since in that case $\Ind(\Cscr)$ and $\Ind(\Dscr)$
    are the same localization of $\Pscr_\Sigma(\Cscr^\heart)$.

    The class $S_\Cscr$ is the strongly saturated class of
    morphisms generated (in the sense
    of~\cite[5.5.4.7]{htt}) by the unit maps $Z\rightarrow UF_\Cscr Z$ as $Z$
    ranges over the objects of $\Pscr_\Sigma(\Cscr^\heart)$.
    Thus, to see that $S_\Cscr=S_\Dscr$, it is enough to prove that
    $Z\rightarrow UF_\Cscr Z$ is in $S_\Dscr$ for all
    $Z\in\Pscr_\Sigma(\Cscr^\heart)$. The opposite inclusion will follow by
    symmetry.
   
    For $Y\in\Cscr^\heart$, the object
    $U_{Y[n]}=U(Y[n])\colon\Cscr^{\heart,\op}\rightarrow\Sp_{\geq 0}$ of
    $\Pscr_\Sigma(\Cscr^\op)$ is the functor
    $$X\mapsto\tau_{\geq 0}\MapSp_\Cscr(X,Y[n]).$$ Here we use that any
    prestable $\infty$-category is naturally enriched in spectra to obtain
    $\MapSp_\Cscr(X,Y[n])$ and then we take the connective cover.
    Note that $U_Y[n]$ is in $\Pscr_\Sigma(\Cscr^\heart)_{[n,n]}$ while
    $U_{Y[n]}$ is in $\Pscr_\Sigma(\Cscr^\heart)_{[0,n]}$. The unit map
    $$U_Y[n]\rightarrow UL_\Cscr(U_Y[n])\we U_{Y[n]}$$ induces an isomorphism on
    degree $n$ homotopy objects:
    $$\pi_n
    U_{Y[n]}(X)\iso\pi_n\Map_{\Cscr}(X,Y[n])\iso\pi_0\Map_\Cscr(X[n],Y[n])\iso\pi_nU_Y[n](X).$$
    By construction, $L_\Cscr U_Y[n]\rightarrow L_\Cscr U_{Y[n]}$ is an
    equivalence and hence $U_Y[n]\rightarrow U_{Y[n]}$ is in $S_\Cscr$ for each
    $n\geq 0$ and $Y\in\Cscr^\heart$. In the $n=1$ case, the cofiber of $U_Y[1]\rightarrow
    U_{Y[1]}$ is $\pi_0U_{Y[1]}$ in $\Pscr_\Sigma(\Cscr^\op)$ and is given by
    $$X\mapsto\Hom_{h\Cscr}(X,Y[1]).$$ In particular, $\pi_0U_{Y[1]}$ is in
    $\Kscr_\Cscr$. Observe that the cofiber of
    $V_Y[1]\rightarrow V_{Y[1]}$ is in $\Kscr_\Dscr$ and is equivalent to 
    $$X\mapsto\Hom_{h\Dscr}(X,Y[1]).$$ By using $F$, we see that
    $\Hom_{h\Cscr}(-,Y[1])\we\Hom_{h\Dscr}(-,Y[1])$ as functors on
    $\Cscr^{\heart,\op}$. It follows that 
    $\pi_0U_{Y[1]}$ is in $\Kscr_\Dscr$. This implies that $U_Y[1]\rightarrow U_{Y[1]}$ is in
    $S_\Dscr$. Continuing in this way, we see that for each $n>0$ the functor
    $$X\mapsto\Hom_{h\Cscr}(X,Y[n])$$ is in $\Kscr_\Cscr$ and in $\Kscr_\Dscr$.
    This implies that each $U_Y[n]\rightarrow U_{Y[n]}$ is in $S_\Dscr$ for
    $n\geq 0$ and $Y\in\Cscr^\heart$: indeed the cofiber is a finite iterated
    extension of the functors $$X\mapsto\Hom_{h\Cscr}(X,Y[n])$$ for $n>0$.

    To summarize the argument of the previous section, we saw that for each
    $Y\in\Cscr^\heart$ and each $n\geq 1$, the functor $X\mapsto\Hom_{h\Cscr}(X,Y[n])$ is an
    object of $\Kscr_\Cscr$. In fact, it is in the heart $\Kscr_\Cscr^\heart$. Then, we
    argued that it is also in $\Kscr_\Dscr^\heart$ using $F$. Finally, the cofiber $C$ of
    $U_Y[n]\rightarrow U_{Y[n]}$ has $$\pi_iC\iso\begin{cases}
    \Hom_{h\Cscr}(-,Y[n-i]) &\text{for $0\leq i<n$,}\\0&\text{otherwise.}\end{cases}$$
    Thus, the cofiber is in $\Kscr_\Dscr$.

    To complete the proof, we show now that for a general object
    $Z\in\Pscr_\Sigma(\Cscr^\heart)$, the unit map $u_Z\colon Z\rightarrow UL_\Cscr Z$ is
    in $S_\Dscr$. Let $\Rscr$ be the full subcategory of
    $\Pscr_\Sigma(\Cscr^\heart)$ on objects $Z$ such that $u_{Z[n]}$
    is in $S_\Dscr$ for all $n\geq 0$.  Since $U$ and $L_\Cscr$ commute with
    finite limits, $\Rscr$ is closed under taking certain fibers. Specifically,
    suppose that $W\xrightarrow{f}X[1]$ is a map where $W$ and $X$ are in
    $\Pscr_\Sigma(\Cscr^\heart)$. If $W$ and $X$ are in $\Rscr$, then so is the
    fiber $Z$ of $f$. Indeed, in this case, the fiber of $f[n]$ is equivalent
    to $Z[n]$ for $n\geq 0$ so the unit map of $Z[n]$ is an equivalence for all $n\geq 0$. This implies that $\Rscr$ is
    closed under extensions: given a cofiber sequence $X\rightarrow
    Z\rightarrow W$ where $W$ and $X$ are in $\Rscr$, we find
    that $Z$ is the fiber of a map $W\rightarrow X[1]$. Moreover, by
    construction, $\Pscr_\Sigma(\Cscr^\heart)\rightarrow\Ind(\Cscr)$ preserves
    compact objects and hence the right adjoint $U$ commutes with
    filtered colimits (see for example~\cite[5.5.7.2]{htt}). It follows that
    $\Rscr$ is closed under filtered colimits in $\Pscr_\Sigma(\Cscr^\heart)$.

    We know by the arguments above that $U_Y[n]\in\Rscr$ for all
    $n\geq 0$ and all $Y\in\Cscr^\heart$. By closure under filtered colimits,
    this implies that $U_Y[n]\in\Rscr$ for $Y\in\Ind(\Cscr^\heart)$ and all
    $n\geq 0$. Now, note that $\Pscr_\Sigma(\Cscr^\heart)$ is $0$-complicial
    and separated and has heart given by
    $\Mod_\Cscr=\Fun^\pi(\Cscr^{\heart,\op},\Mod_\ZZ)$, the abelian category of
    product preserving functors $\Cscr^{\heart,\op}\rightarrow\ZZ$. Hence,
    $\Pscr_\Sigma(\Cscr^\heart)\we\Dscr(\Mod_\Cscr)_{\geq 0}$.
    Since $\Mod_\Cscr$ has enough projective objects (given by the
    representable functors $U_Y$), 
    every object $Z\in\Pscr_\Sigma(\Cscr^\heart)$ admits an increasing
    exhaustive filtration $F_\star Z$ where $F_iZ=0$ for $i<0$ and
    $\gr_i^FZ\in\Ind(\Cscr^\heart)[i]$ for $i\geq 0$. We could even take a filtration with
    graded pieces shifted projective, but we do not need this here. (See also
    \cite[5.5.8.14]{htt}.) Write $\gr_i^FZ\we
    U_{Y_i}[i]$ for some $Y_i\in\Ind(\Cscr^\heart)$. By closure under
    extensions, inducting on $i\geq 0$, we see that each $F_iZ$ is in $\Rscr$ for finite $i$. Finally, by closure
    under filtered colimits, $\colim_iF_iZ\we Z$ is in $\Rscr$. This completes
    the proof.
\end{proof}

If $\Ascr$ is a small abelian category, then $\Dscr^b(\Ascr)_{\geq 0}$ is
$0$-complicial, so Corollary~\ref{cor:1} also follows from Theorem~\ref{thm:3}.
Corollary~\ref{maincor:3} follows immediately from Theorem~\ref{thm:3}.

\begin{corollary}\label{cor:3}
    If $X$ is quasi-compact, quasi-separated, and $0$-complicial, then $\Perf(X)$ admits a unique
    $\infty$-categorical enhancement.
\end{corollary}

To find $X$ which is not separated but where $\Perfscr(X)_{\geq 0}$ is
$0$-complicial, consider the case of a regular but not separated scheme as in
the next example.

\begin{example}
    Let $X=\widetilde{\AA}^2$ denote the affine plane with the origin doubled. The scheme
    $X$ is quasi-compact and quasi-separated but is not semi-separated. It certainly does
    not have enough locally free sheaves. In fact, the category of locally free sheaves on
    $X$ is equivalent to the category of locally free sheaves on $\AA^2$ via pullback along the collapse
    map $X\rightarrow\AA^2$. On the other hand, $X$ is smooth, so that
    $\Perf(X)\we\D^b(\Coh(X))$. It follows that $X$ is $0$-complicial. In this
    case, uniqueness of enhancements of $\Perf(X)$ follows from
    Corollary~\ref{cor:1}.
\end{example}

Corollary~\ref{maincor:4} also follows directly from Theorem~\ref{thm:3}.

\begin{corollary}\label{cor:4}
    If $\Ascr$ is a Grothendieck abelian category such that $\D(\Ascr)$ is compactly
    generated and $\Dscr(\Ascr)_{\geq 0}\cap\Dscr(\Ascr)^\omega$ is $0$-complicial, then
    $\Dscr(\Ascr)^\omega$ admits a unique $\infty$-categorical enhancement.
\end{corollary}

Finally, we prove Corollary~\ref{maincor:5}.

\begin{corollary}\label{cor:5}
    If $\Ascr$ is a small abelian category, then
    $\D^-(\Ascr)$ and $\D^+(\Ascr)$ admit unique $\infty$-categorical enhancements.
\end{corollary}

\begin{proof}
    If $\Ascr$ is a small abelian category, then $\Dscr^-(\Ascr)_{\geq 0}$ is
    $0$-complicial by using brutal truncations. Since
    $\Dscr^-(\Ascr)\we\SWscr(\Dscr^-(\Ascr)_{\geq 0})$, it follows from
    Theorem~\ref{thm:3} that $\D^-(\Ascr)$ admits a unique $\infty$-categorical
    enhancement. As $\Dscr^+(\Ascr)\we\left(\Dscr^-(\Ascr^\op)\right)^\op$, we
    see that $\Dscr^+(\Ascr)$ admits a unique $\infty$-categorical enhancement.
\end{proof}

\section{Discussion of the meta theorem}\label{sec:meta}

We briefly discuss Meta Theorem~\ref{meta:dg}. In general, one wants to simply
say the words ``all of our proofs now work $k$-linearly for any commutative
connective $\EE_\infty$-algebra $k$''. Applying this to the case where $k=\ZZ$,
we would obtain the previous results on the uniqueness of dg enhancements,
since pretriangulated dg categories over $\ZZ$ are equivalent to $\ZZ$-linear
stable $\infty$-categories.

However, we need to be more careful. Indeed, the heart of any stable
$\infty$-category with a $t$-structure or any prestable $\infty$-category is
automatically an additive category and is hence $\ZZ$-linear. Our proofs in many
places construct functors from this $1$-categorical information. If we want to
check that those functors are themselves $\ZZ$-linear, we need to do a little
more work.

\begin{theorem}\label{thm:canonical}
    If $\Cscr$ is an anticomplete or separated $0$-complicial Grothendieck prestable $\infty$-category,
    then the stable $\infty$-category $\Sp(\Cscr)$ admits a canonical
    $\ZZ$-linear structure.
\end{theorem}

Theorem~\ref{thm:canonical} implies Meta Theorem~\ref{meta:dg}
because it shows that that a $\ZZ$-linear structure on a stable presentable
$\infty$-category with a separated or anticomplete $0$-complicial $t$-structure is not extra structure.
It also applies to the results in the cases of the small categories, as in
Corollary~\ref{maincor:1} or~\ref{maincor:3}, since the proofs pass through anticomplete
$0$-complicial or separated $0$-complicial Grothendieck prestable $\infty$-categories.

\begin{proof}[Proof of Theorem~\ref{thm:canonical}]
    We will first give the proof in the separated case.
    There is an adjunction
    $$\Dscr(-)_{\geq
    0}\colon\Groth_0^\lex\leftrightarrows\Groth_\infty^{\lex,\sep}\colon(-)^\heart,$$
    where the right adjoint is fully faithful and has essential image the full subcategory of
    $\Groth_\infty^\lex$ on the separated $0$-complicial Grothendieck prestable
    $\infty$-categories (see~\cite[C.5.4.5]{sag}).

    The category $\Groth_0^\lex$ is a symmetric monoidal category with unit
    $\Mod_\ZZ$. To see this, one uses~\cite[C.5.4.16]{sag}, which implies that if
    $\Ascr$ and $\Bscr$ are Grothendieck abelian categories, then so is
    $\Ascr\otimes\Bscr$, where the tensor product is computed in $\PrL$, the
    $\infty$-category of presentable $\infty$-categories and left adjoint functors.
    This gives a symmetric monoidal structure on $\Groth_0$ and it can be
    restricted to the subcategory $\Groth_0^\lex$ by the argument of~\cite[C.4.4.2]{sag}.
    With respect to the symmetric monoidal structures, $(-)^\heart$ is
    symmetric monoidal and the left adjoint $\Dscr(-)_{\geq 0}$
    is then naturally oplax symmetric monoidal. This presents some
    problems and means that we cannot use the most naive argument to give the
    proof of Theorem~\ref{thm:canonical}.

    The fact that $\Dscr(-)_{\geq 0}$ is oplax symmetric monoidal implies that
    $\Dscr(\ZZ)_{\geq 0}$ is not a commutative algebra
    object in $\Groth_\infty^{\lex,\sep}$ but rather an $\EE_\infty$-coalgebra object.
    This may seem a little strange, but consider the fact that the natural
    multiplication map
    $$\Dscr(\ZZ\otimes_\SS\ZZ)_{\geq 0}\we\Dscr(\ZZ)_{\geq 0}\otimes\Dscr(\ZZ)_{\geq 0}\rightarrow\Dscr(\ZZ)_{\geq
    0}$$
    is not in $\Groth_\infty^{\lex,\sep}$ as it is not left exact. Indeed, it takes
    $\ZZ$ in the heart of the left hand side to
    $\THH(\ZZ)\we\ZZ\otimes_{\ZZ\otimes_\SS\ZZ}\ZZ$ on the right hand
    side. Since $\THH(\ZZ)$ has non-zero homotopy groups in arbitrarily high
    degrees by~\cite{bokstedt}, it is not in the heart.

    The fact that the left adjoint is oplax implies that for any Grothendieck abelian
    category $\Ascr$, the Grothendieck prestable $\infty$-category $\Dscr(\Ascr)_{\geq 0}$
    is a comodule for $\Dscr(\ZZ)_{\geq 0}$ in $\Groth_\infty^{\lex,\sep}$. To see this,
    note that there is a natural equivalence $$\Ascr\we\left(\Dscr(\ZZ)_{\geq 0}^{\otimes
    n}\otimes\Dscr(\Ascr)_{\geq 0}\right)^\heart$$
    and hence by adjunction a natural left exact left adjoint functor $$\Dscr(\Ascr)_{\geq 0}\rightarrow\Dscr(\ZZ)_{\geq
    0}^{\otimes n}\otimes\Dscr(\Ascr)_{\geq 0}$$
    for all $n$. It is not hard to see that these functors assemble into the structure of a
    $\Dscr(\ZZ)_{\geq 0}$-comodule on $\Dscr(\Ascr)_{\geq 0}$.
    We will prove that the comodule structure is naturally right adjoint to a
    module structure in $\Groth_\infty$.

    Consider for simplicity for a
    moment the case of $\Dscr(R)_{\geq 0}$ where $R$ is some ring. Then, the functor
    $\Dscr(R)_{\geq 0}\rightarrow\Dscr(\ZZ)_{\geq 0}\otimes\Dscr(R)_{\geq
    0}\we\Dscr(\ZZ\otimes_\SS R)_{\geq 0}$ is the left exact functor given by
    restriction of scalars along the map $\ZZ\otimes_\SS R\rightarrow R$. 
    In particular, it admits a left adjoint itself $$\Dscr(\ZZ)_{\geq
    0}\otimes\Dscr(R)_{\geq 0}\rightarrow\Dscr(R)_{\geq 0}.$$ This left adjoint is typically not left
    exact. It is easy to see using the functoriality of adjoints that this makes $\Dscr(R)_{\geq 0}$ into a
    $\Dscr(\ZZ)_{\geq 0}$-module in $\Groth_\infty$ and hence by taking
    spectrum objects we obtain a canonical $\Dscr(\ZZ)$ action on $\Dscr(R)$.
    (Note that technically we should also discuss the left adjoints to the
    maps $\Dscr(R)_{\geq 0}\rightarrow\Dscr(\ZZ)_{\geq 0}^{\otimes
    n}\otimes\Dscr(R)_{\geq 0}$. The argument is the same as the $n=1$ case
    here and in the next paragraph,
    so we omit it.)

    Now, suppose that $\Cscr$ is a general separated $0$-complicial Grothendieck prestable
    $\infty$-category. The important thing is to check that
    $H\colon\Cscr\rightarrow\Dscr(\ZZ)_{\geq 0}\otimes\Cscr$ preserves all limits so that
    it admits a left adjoint. Choose a generator $X\in\Cscr^\heart$ and let
    $R=\Hom_\Cscr(X,X)$. By the $\infty$-categorical Gabriel--Popescu
    theorem~\cite[C.2.1.6]{sag}, we have that the natural fully faithful functor
    $V=\Map_\Cscr(X,-)\colon\Dscr(R)_{\geq 0}$ admits a left exact left adjoint
    $E\colon\Dscr(R)_{\geq 0}\rightarrow\Cscr$. We claim that the following
    diagram
    $$\xymatrix{\Dscr(R)_{\geq 0}\ar[r]^<<<<<{G}\ar[d]^E&\Dscr(\ZZ)_{\geq
    0}\otimes\Dscr(R)_{\geq 0}\we\Dscr(\ZZ\otimes_\SS R)_{\geq 0}\ar[d]^F\\
    \Cscr\ar[r]^H&\Dscr(\ZZ)_{\geq 0}\otimes\Cscr}$$
    is right adjointable. In other words, if we let $V$ be the fully faithful right adjoint to $E$
    and $U$ be right adjoint to $F$, then there is an equivalence of functors
    $G\circ V\we U\circ H$. Note that $U\we\Map_{\Dscr(\ZZ)_{\geq
    0}\otimes\Cscr_{\geq 0}}(\ZZ\otimes X,-)$ is fully faithful. Pick $Y\in\Cscr$.
    There are natural equivalences
    \begin{align*}
        UHY&\we UHEVY\we UFGVY\\
        &\we GVY,
    \end{align*}
    which is what we wanted to show.

    In particular, the adjointability of the diagram together with the
    conservativity of $U$ implies that $H$ preserves limits, as desired.
    It follows that $\Cscr$ is a canonically a $\Dscr(\ZZ)_{\geq 0}$-module in
    $\Groth_\infty$ and hence that $\Sp(\Cscr)$ is canonically a
    $\Dscr(\ZZ)$-module in $\mathrm{Pr}^{\mathrm{L}}_{\mathrm{st}}$, the $\infty$-category
    of stable presentable $\infty$-categories and left adjoint functors.

    The proof is the same in the anticomplete case, but where we use~\cite[C.5.8.12, C.5.8.13]{sag} to
    write a general anticomplete $0$-complicial Grothendieck prestable $\infty$-category as
    a left exact localization of a separated $0$-complicial Grothendieck prestable
    $\infty$-category.
\end{proof}

\section{(Counter)examples, questions, and conjectures}\label{sec:examples}

We discuss a wide range of ideas in this section.
Section~\ref{sub:completeness} discusses the question of when $\Dscr(\Ascr)$ is
left complete and of when $\widehat{\D}(\Ascr)$ admits a unique enhancement.
Section~\ref{sub:qcoh} is about what is not known for derived categories of
quasi-coherent sheaves. In Section~\ref{sub:sing}, we relate our work to
singularity categories. Section~\ref{sub:nstable} is about the conjectural
theory of stable $n$-categories.
In the spirit of all papers on triangulated categories and dg categories, Section~\ref{sub:conjecture} discusses a foolishly optimistic
conjecture. Finally, the brief Section~\ref{sub:categories} is about some
categorical questions which would make all of our proofs easier and strengthen
our results.

We include several updates due to the work~\cite{cns} of Canonaco, Stellari, and Neeman.

\subsection{Completeness and products}\label{sub:completeness}

\begin{question}
    Let $\Ascr$ be a Grothendieck abelian category.
    Is it true that every (possibly presentable) enhancement $\Cscr$ of $\widehat{\D}(\Ascr)$ is
    equivalent to $\wDscr(\Ascr)$?
\end{question}

\begin{remark}
    This has now been proved by Canonaco, Stellari, and Neeman
    in~\cite[Theorem~A]{cns}.
\end{remark}

Lurie effectively proved that when countable products are exact in $\Ascr$, then
$\D(\Ascr)\we\wD(\Ascr)$, so in that case a positive answer in the presentable
case is given by Theorem~\ref{mt:2} and in general by Theorem~\ref{thm:appendix}.

\begin{definition}
    Let $\Ascr$ be a Grothendieck abelian category.
    \begin{enumerate}
        \item[(a)] We say that $\Ascr$ is AB4* if products in $\Ascr$ are right
            exact.
        \item[(b)] We say that $\Ascr$ is AB4*($\omega$) if countable products
            in $\Ascr$ are right exact.
        \item[(c)] We say that $\Ascr$ is AB4*$n$ if the derived functors
            $\prod_I^i$ vanish for $i>n$ and all indexing sets $I$.
        \item[(d)] We say that $\Ascr$ is AB4*$n(\omega)$ if the derived
            product
            functors $\prod_I^i$ vanish for $i>n$ and all countable indexing
            sets $I$.
    \end{enumerate}
\end{definition}

Condition AB4* is satisfied for example by $\Mod_A$ where $A$ is any
associative ring. It is definitely not true in general, as examples below
illustrate.

\begin{lemma}
    If $\Ascr$ is a Grothendieck abelian category that satisfies {\em
    AB4*($\omega$)}, then $\Dscr(\Ascr)$ is left complete.
\end{lemma}

\begin{proof}
    It suffices by Proposition~\ref{prop:sepcomplete} to check $\Dscr(\Ascr)_{\geq 0}$ is
    closed under countable products in $\Dscr(\Ascr)$. Let $\{X(i)\}$ be a countable collection of
    objects in $\Dscr(\Ascr)_{\geq 0}$, which we represent as $X(i)_\bullet$ for some fibrant
    complexes in $\Ascr$. The product is represented by $\prod_i X(i)_\bullet$. We have to
    prove that $\H_i(\prod_i X(i)_\bullet)=0$ for $i<0$. But, by AB4*($\omega$), the
    homology of a product is the product of the homologies, so this is clear.
\end{proof}

\begin{example}\label{ex:almost}
    Categories of almost modules are AB4*. If $A$ is an associative
    ring with a $2$-sided ideal $I$ such that $I^2=I$, then the almost category $\Mod_A^{a}$
    satisfies AB4* by work of Roos. It follows that $\Dscr(\Mod_A^{a})$ is left complete.
    Additionally, $\D(\Mod_A^a)$ admits a unique presentable $\infty$-categorical
    enhancement by Theorem~\ref{mt:2}.
\end{example}

\begin{proposition}
    In the situation of Example~\ref{ex:almost}, $\D(\Mod_A^a)$ admits a unique
    $\infty$-categorical enhancement.
\end{proposition}

This case is interesting because typically $\Mod_A^a$ is typically not compactly generated and
in fact does not even admit any non-zero projective objects!
See~\cite[Theorem~4.1]{roos-inverse} and Lemma~\ref{lem:compacts} below.

\begin{proof}
    We can see from Lemma~\ref{lem:detection} that if $\Cscr$ is a model for $\D(\Mod_A^a)$,
    then the $t$-structure on $\Cscr$ is compatible with countable products. It is also
    left separated so it is in fact left complete. However, using Proposition~\ref{prop:ba},
    we find that $\Cscr_{\leq 0}\we\Dscr(\Mod_A^a)_{\leq 0}$. Completeness of $\Cscr$ and
    $\Dscr(\Mod_A^a)$ now implies that $\Cscr\we\Dscr(\Mod_A^a)$.
\end{proof}

\begin{remark}
    What we see more generally is that if $\Dscr(\Ascr)$, then Theorem~\ref{mt:2} can be strengthened to say
    that there is a unique $\infty$-categorical enhancement of $\D(\Ascr)$.
\end{remark}

\begin{lemma}\label{lem:compacts}
    Suppose that $R$ is a local ring and $I\subseteq R$ a proper flat ideal $I$ such that $I^2=I$.
    Then, the only compact object of $\Mod_R^a$ is the zero object. 
    Similarly, the only compact object of $\D(\Mod_R^a)$ is the zero object.
\end{lemma}

\begin{proof}
    Since $j^*\colon\Mod_R\rightarrow\Mod_R^a$ preserves filtered colimits, the left
    adjoint $j_!$ preserves compact objects. But, by definition, $j_!M$ is an $R$-module
    such that $Ij_!M=j_!M$. Since $j_!M$ is compact, we see that it is finitely presented. But,
    $I$ is contained in the Jacobson radical of $R$, so $Ij_!M=j_!M$ implies that $j_!M=0$.
    The proof in the derived category case is the same, using that $I$ is flat
    and that the bottom homotopy group of a perfect complex of $R$-modules is finitely presented.
\end{proof}

The axiom AB4* is not satisfied in general. The original example is due to
Grothendieck~\cite{tohoku}.

\begin{example}
    Let $X$ be a topological space and let $\Shv(X)$ be the abelian category of sheaves of
    abelian groups on $X$. Then, $\Shv(X)$ is Grothendieck abelian, but it typically does
    not satisfy AB4* or even AB4*($\omega$). The reason is that products are computed on stalks, but the
    restriction functors do not generally preserve products. Write $\PShv(X)$ for the
    category $\Fun(\Op(X)^{\op},\Mod_\ZZ)$ of presheaves of abelian groups. Then, $\Shv(X)$
    is a left exact localization of $\PShv(X)$. In particular, the inclusion functor
    preserves arbitrary products. But, it is not right exact in general. Thus, consider a
    collection $\{0\rightarrow\Fscr_i\rightarrow\Gscr_i\rightarrow\Hscr_i\rightarrow
    0\}_{i\in I}$ of exact sequences of sheaves of abelian groups. We can compute the product sequence
    $$0\rightarrow\prod_i\Fscr_i\rightarrow\prod_i\Gscr_i\rightarrow\prod_i\Hscr_i,$$
    which is exact on the left since products are always left exact.
    Each $\prod_i\Fscr$ is the sheaf with values $(\prod_i\Fscr_i)(U)\iso\prod_i\Fscr_i(U)$,
    where the latter term is computed as the product in abelian groups. The question is
    whether or not the sequence above is exact on the right, or simply whether in this case
    $\prod_i\Gscr_i\rightarrow\prod_i\Hscr_i$ is surjective as a map of sheaves. Note
    however, that the maps $\Gscr_i(U)\rightarrow\Hscr_i(U)$ are typically not surjective
    for any given $U$. Let $X$ be a space and $x\in X$ a point with a strictly decreasing family of open neighborhoods
    $\cdots\subset U_2\subset U_1$ with intersection $\{x\}$. Write $j(k)$ for
    the inclusion of $U_k$ in $X$ and $i$ for the inclusion of $\{x\}$ in $X$.
    Consider the natural transformations $j(k)_!\ZZ_{U_k}\rightarrow i_*\ZZ_{x}$,
    where $\ZZ_{U_k}$ is the constant sheaf associated to $\ZZ$ on $U_k$ and $\ZZ_{x}$ is
    the constant sheaf $\ZZ$ on $\{x\}$.
    Each of these maps is surjective. Now, consider the map
    $\prod_k j(k)_!\ZZ_{U_k}\rightarrow\prod_k i_*\ZZ_x$. The right hand term is evidently
    non-zero. But, the product on the left is actually the zero sheaf if $\{x\}$ is not open!
\end{example}

Similarly, the Grothendieck abelian category of quasi-coherent sheaves on a scheme $X$ is typically not AB4*($\omega$). Indeed, Roos has
shown that if $U=\Spec R-\{\mathfrak{m}\}$ where $R$ is a noetherian local ring of Krull
dimension $d$ and $\mathfrak{m}$ is the maximal ideal, then products in $\QCoh(U)$ are
not exact if $d\geq 2$. A more precise statement can be made.

\begin{example}
    Roos showed in~\cite[Theorem~1.5]{roos-inverse} that
    if $U=\Spec R-\{\mathfrak{m}\}$ is the punctured spectrum of a local ring as above, then
    $\QCoh(U)$ is AB4*$(d-1)$ and this is the best possible, meaning that $\QCoh(U)$ is not
    AB4*$n$ for any $n<d-1$.
\end{example}

\begin{example}
    For a specific example, let $R=k[x,y]_{(0,0)}$ with $\mathfrak{m}=(x,y)$ and $d=2$. Let
    $X=\Spec R$ and $U=X-\{\mathfrak{m}\}$.
    Consider the countable product $\prod_\NN\Oscr_U$ in $\Dscr(\QCoh(U))$. Since $j_*$
    preserves products, we have that $$j^*\prod_\NN j_*\Oscr_U\we
    j^*j_*\prod_\NN\Oscr_U\we\prod_\NN\Oscr_U.$$ We can compute the homology groups of
    $\prod_\NN j_*\Oscr_U$ as $$\H_i\left(\prod_\NN j_*\Oscr_U\right)\iso\prod_\NN
    \R^{-i}j_*\Oscr_U\iso\begin{cases}
    \prod_\NN R&\text{if $i=0$,}\\
    \prod_\NN\H^{-i-1}_{\mathfrak{m}}(R)&\text{if $i\neq 0$.}
    \end{cases}$$
    The local cohomology group $\H^{n}_{\mathfrak{m}}(R)$ in this case is $K/R$ if
    $n=2$ and zero otherwise. In particular, we see that $$\H_{-1}\left(\prod_\NN
            j_*\Oscr_U\right)\iso\prod_\NN K/R.$$ Each element of $K/R$ is killed by some
    power of $\mathfrak{m}$. However, this is not true of $\prod_{\NN}K/R$.
    Thus, the localization
    $j^*\prod_\NN K/R$ is non-zero. Since $\Dscr(\QCoh(X))\rightarrow\Dscr(\QCoh(U))$ is
    $t$-exact we see that $\H_{-1}\left(\prod_\NN\Oscr_U\right)$ is non-zero in $\Dscr(\QCoh(U))$ and
    that the $t$-structure on $\Dscr(\QCoh(U))$ is not compatible with countable products.
    Similarly, $\QCoh(U)$ is AB4*$1(\omega)$, but not AB4*$0(\omega)$.
\end{example}

\begin{remark}
    Kanda has recently shown in~\cite{kanda} that for a noetherian scheme $X$ with an ample family of
    line bundles, $X$ satisfies AB4* if and only if $X$ is affine.
\end{remark}

We see however that any such punctured spectrum $X$ is quasi-compact and separated. In
particular, we have that $$\Dscr(\QCoh(X))\we\wDscr(\QCoh(X))\we\Dscr_{\qc}(X),$$
so that $\Dscr(\QCoh(X))$ is left complete. In particular, we see that separated plus left
complete does not imply AB4*.

\begin{question}\label{q:ab4}
    Let $X$ be a quasi-compact and quasi-separated scheme. Does $\QCoh(X)$ satisfy AB4*$n$ for some
    finite $n$?
\end{question}

\begin{proposition}\label{prop:complete}
    Suppose that $\Ascr$ is {\em AB4*$n(\omega)$} for some $n$. Then, $\Dscr(\Ascr)$ is left complete.
\end{proposition}

Akhil Mathew pointed out a gap in our original proof, which we have fixed.
Since then, Jack Hall pointed out that the new reasoning is quite similar to the
argument in the proof of~\cite[\href{https://stacks.math.columbia.edu/tag/07K7}{Tag
07K7}]{stacks}.

Note that~\cite[Remark~1.2]{neeman-left} implies that the abelian category $\Ascr=\QCoh(\B\Ga)$
is not AB4*$n(\omega)$ for any finite $n$. See also Example~\ref{ex:neeman}.

\begin{proof}
We begin by proving the following in $\D(\Ascr)$: if $\{X(i)\}$ is a countable
collection of objects of $\D(\Ascr)_{\geq 0}$, then $\prod_i X(i)$ is
$(-n)$-connective. If this is the case, then a countable cofiltered limit of
connective objects will be $(-n-1)$-connective.

To compute the product, we first find a fibrant replacement $F(i)_*$ for each
$X(i)$. Typically, $F(i)_*$ is typically neither bounded above or
below, but it is some chain complex of injective objects of A. The product
$\prod_i X(i)$ is `computed’ by the chain complex $\prod_i F(i)_*$ by using the
injective model category structure on $\D(\Ascr)^I$.

There is a tiny subtlety here: in general it is difficult to control the
fibrant objects of functor categories in the injective model category
structure. But, since $I$ is discrete, the fibrant objects of the injective
model category structure on $\D(\Ascr)^I$ are those which are pointwise fibrant in
$\D(\Ascr)$.

Now, let $\sigma_{\leq 0}$ denote the bad truncation at $0$, so that there is a
    natural map $\sigma_{\leq 0}F(i)_* \rightarrow F(i)_*$ for each i. Since
    $\sigma_{\leq 0}F(i)_*$ is a bounded above complex of injectives, it is
    fibrant. Moreover, the map $\sigma_{\leq 0}F(i)_* \rightarrow F(i)_*$ is a $0$-equivalence.

To prove that $\prod_i X(i)$ is $(-n)$-connective, it is now enough to prove
    that $\sigma_{\leq 0}(\prod_i F(i)_*)$ is $(-n)$-connective.

However, $\sigma_{\leq 0}(\prod_i F(i)_*)=\prod_i (\sigma_{\leq 0} F(i)_*)$. This is an
isomorphism of chain complexes of injective objects of $A$. By our hypothesis
on the $X(i)$, $\sigma_{\leq 0}F(i)_*$ is an injective resolution of its top
homology (which sits in degree $0$). It follows that the homology of $\prod_i
(\sigma_{\leq 0} F(i)_*)$ computes the right derived functors of the functor
``product over I'' of $\{\H_0(\sigma_{\leq 0}F(i)_*)\}$ and in particular, since we
assume AB4*$n(\omega)$, it follows that $\H_m(\prod_i (\sigma_{\leq 0} F(i)_*))=0$
for $m<-n$. But, this means that $\H_m(\sigma_{\leq 0}(\prod_i F(i)_*))=0$ for
$m<-n$, which is what we wanted to show.

    Now we show that Postnikov towers converge. Fix
    $X\in\Dscr(\Ascr)_{\geq 0}$ and consider for each $m\geq 0$ the fiber
    sequence $\tau_{\geq m+1}X\rightarrow X\rightarrow\tau_{\leq m}X$.
    Taking the limit over $m$ we get a fiber sequence $$\lim_m\tau_{\geq
    m+1}X\rightarrow X\rightarrow\lim\tau_{\leq m}X.$$ To show that the
    Postnikov tower converges, it is enough to prove that $$\lim_m\tau_{\geq
    m+1}X\we 0.$$ We can start this limit at any point we want and thus assume
    it is a limit of $r$-connective objects for $r$ any given integer. Thus,
    the limit is $(r-n)$-connective for any $r$ by the argument above, and
    hence the limit vanishes since $\D(A)$ is left separated. The same argument will show that
    every tower is a Postnikov tower. Consider a tower $\{X(m)\}$ where
    $X(m)\in\Dscr(\Ascr)_{[0,m]}$ and $X(m)\rightarrow X(m-1)$ induces an
    equivalence $\tau_{\leq m-1}X(m)\we X(m-1)$. Fix $r\in\NN$. Then, we have fiber
    sequences $\tau_{\geq r+1}X(m)\rightarrow X(m)\rightarrow X(r)$ and $$\lim_m\tau_{\geq r+1}X(m)\rightarrow \lim_m X(m)\rightarrow X(r)$$ is a
    fiber sequence. We see from the argument above that the leftmost term is
    $(r+1-n)$-connective. Hence, $\pi_i\lim_m X(m)\iso\pi_iX(r)$ for $i<r-n$.
    Thus, $\tau_{\leq r-n-1}\lim_m X(m)\we \tau_{\leq r-n-1}X(r)\we X(r-n-1)$. Since $r$ was again
    chosen to be arbitrary, we see that the Postnikov tower associated to
    $\lim_n\tau_{\leq n}X(m)$ is again the tower $\{X(m)\}$.
\end{proof}

Akhil Mathew asked the following question.

\begin{question}
    Suppose that $\Cscr_{\geq 0}$ is a Grothendieck prestable $\infty$-category
    which is $m$-complicial for some finite $m$ and such that $\Cscr_{\geq
    0}^\heart$ satisfies AB4*$n(\omega)$ for some $n$. Is $\Cscr_{\geq 0}$ complete?
    Proposition~\ref{prop:complete} says that the answer is `yes' for $m=0$.
\end{question}

\subsection{Quasi-coherent sheaves}\label{sub:qcoh}

Let $X$ be a quasi-compact and quasi-separated scheme. Then, $\Dscr_{\qc}(X)$ is
a compactly generated stable presentable $\infty$-category with an accessible,
left and right complete $t$-structure which is additionally compatible with
filtered colimits. The heart is $\QCoh(X)$, but in general the natural map
$\Dscr(\QCoh(X))\rightarrow\Dscr_{\qc}(X)$ is not an equivalence or even fully
faithful when applied to bounded objects. See~\cite[Expos\'e~II,
Appendice~I]{sga6} for a counterexample of Verdier.

So, one open problem is whether or not $\D_{\qc}(X)$ admits a unique
$\infty$-categorical enhancement. Because of compact generation, any such will
be presentable.

\begin{question}\label{q:omni} Let $X$ be a quasi-compact and quasi-separated scheme.
    \begin{enumerate}
        \item[(i)]   Does $\D(\QCoh(X))$ admit a unique $\infty$-categorical model?
        \item[(ii)]   Does $\wD(\QCoh(X))$ admit a unique $\infty$-categorical model?
        \item[(iii)]   Does $\wD(\QCoh(X))$ admit a unique presentable
            $\infty$-categorical model?
        \item[(iv)]   Is the natural map $\D(\QCoh(X))\rightarrow\wD(\QCoh(X))$ an
            equivalence?
        \item[(v)] Does $\D_\qc(X)$ admit a unique $\infty$-categorical model?
        \item[(vi)] Does $\Perf(X)$ admit a unique $\infty$-categorical model?
    \end{enumerate}
\end{question}

\begin{remark}
We asked in Question~\ref{q:ab4} if $\QCoh(X)$ satisfies AB4*$n(\omega)$ for some
$n$ when $X$ is quasi-compact and quasi-separated. If so,
$\D(\QCoh(X))\we\wD(\QCoh(X))$.
\end{remark}

The answer to these questions is ``yes'' if $X$ has affine diagonal, in which
case $\D(\QCoh(X))\we\wD(\QCoh(X))\we\D_\qc(X)$. When $X$ is
instead a stack, then there are known cases where $\D(\QCoh(X))$ differs from
$\wD(\QCoh(X))$. Indeed, it is noted in~\cite[Remark~C.4]{hall-neeman-rydh} that if $X$ is
quasi-compact with affine diagonal, then $\D^+(\QCoh(X))\rightarrow\D_\qc^+(X)$
is fully faithful and extends to an equivalence
$\wD(\QCoh(X))\we\D_\qc(X)$. This applies in particular to $\B\Ga$;
Neeman proved in~\cite{neeman-left} that $\Dscr(\QCoh(\B\Ga))$ is not left
complete. So, in this particular case,
$\D(\QCoh(\B\Ga))\rightarrow\wD(\QCoh(\B\Ga))$ is not an equivalence.

Several more cases of Question~\ref{q:omni} have now been settled in~\cite{cns}
by Canonaco, Neeman, and Stellari. Positive answers to (i), (ii), and (iii)
follow from~\cite[Theorem~A]{cns} together with the Meta-Theorem. Additionally,
Canonaco, Neeman, and Stellari prove that there are unique dg
enhancements of $\Perf(X)$ and $\D_\qc(X)$ (settling the dg analogue of (v) and
(vi)). As far as we can see the case of
unique $\infty$-categorical enhancements in (v) and (vi) remains open because
$\D_\qc(X)$ need not be $0$-complicial in general (see
Example~\ref{ex:not0complicial}) and so the Meta-Theorem
does not apply.

In the remainder of this section, we will explore one possible route
to prove that $\D_\qc(X)$ admits a unique $\infty$-categorical
enhancement.

Let $\Perfscr(X)_{\geq 0}=\Perfscr(X)\cap\Dscr_{\qc}(X)_{\geq 0}$.
Unless $X$ satisfies some kind of regularity hypotheses,
$\Perfscr(X)_{\geq 0}$ will not be part of a $t$-structure on $\Perfscr(X)$.
By~\cite[Lemma~2.7]{bhatt-tannaka}, $$\Ind(\Perfscr(X)_{\geq
0})\we\Dscr(X)_{\geq 0}.$$
Now, consider Questions~\ref{q:omni}(v) and (vi). One might be tempted to argue as we did
for Theorem~\ref{mt:3} to prove (vi) and then deduce (v) from this via the
following straightforward argument.

\begin{lemma}
    Let $X$ be a quasi-compact and quasi-separated scheme. If $\Perf(X)$ admits
    a unique $\infty$-categorical enhancement, so does $\D_\qc(X)$.
\end{lemma}

\begin{proof}
    Suppose that $\Cscr$ is an $\infty$-categorical model for $\D_\qc(X)$.
    Then, $\Cscr$ is presentable, compactly generated, and $\Cscr^\omega$ is an
    $\infty$-categorical model for $\Perf(X)$. Hence, by hypothesis,
    $\Cscr^\omega\we\Perfscr(X)$, whence
    $\Cscr\we\Ind(\Cscr^\omega)\we\Ind(\Perfscr(X))\we\Dscr_\qc(X)$, as
    desired.
\end{proof}

Unfortunately, proving that $\Perf(X)$ admits a unique $\infty$-categorical model
is out of our reach at the moment since we do not know if $\Perfscr(X)_{\geq 0}$ is $0$-complicial
in general, so we cannot appeal to Theorem~\ref{mt:3}. Note that there are a
priori more
objects of $\Perfscr(X)_{\geq 0}^\heart$ than of $\Perfscr(X)_{\geq
0}\cap\Dscr_\qc(X)^\heart$. In particular, even if $X$ does not have enough
perfect quasi-coherent sheaves, it might still be $0$-complicial. We hope to
return to this and the next question and conjecture in future work.

\begin{question}
    For which quasi-compact and quasi-separated schemes $X$ is
    $\Perfscr(X)_{\geq 0}$ $0$-complicial?
\end{question}

\begin{conjecture}
    Let $X$ be quasi-compact and quasi-separated. Then, $\Perfscr(X)_{\geq 0}$
    is $n$-complicial for some $n$.
\end{conjecture}

The idea would be to use~\cite[Proposition~B.11]{thomason-trobaugh}, which says
that there exists an integer $n$ such that $\H^i(X,\Fscr)=0$ for $i>n$ and all quasi-coherent sheaves $\Fscr$ on $X$.
We prove in the following remark that if $X$ is quasi-compact and
quasi-separated, then $\Dscr_\qc(X)_{\geq 0}$ is typically not $0$-complicial.

\begin{remark}
    Let $\Groth_\infty^{\lex,\sep}$ denote the $\infty$-category of separated
    Grothendieck prestable $\infty$-categories and left exact left adjoint
    functors.
    We also have $\Groth_0^{\lex}$,
    the category of Grothendieck abelian categories and left exact colimit
    preserving functors between them. Finally, we have
    $\Groth_\infty^{\lex,\comp}$, the $\infty$-category of complete
    Grothendieck prestable $\infty$-categories and left exact left adjoint
    functors. There is a tripod of fully faithful left adjoint functors
    $$\xymatrix{&\Groth_0^\lex\ar[dr]^{\wDscr(-)_{\geq 0}}\ar[dl]_{\Dscr(-)_{\geq 0}}&\\
    \Groth_\infty^{\lex,\sep}&&\Groth_\infty^{\lex,\comp}.}$$
    The essential image of $\Groth_0^\lex$ in $\Groth_\infty^{\lex,\sep}$ is
    the full subcategory of $0$-complicial separated Grothendieck prestable
    $\infty$-categories. The essential image of $\Groth_0^\lex$ in
    $\Groth_\infty^{\lex,\comp}$ is the full subcategory of weakly
    $0$-complicial complete Grothendieck prestable $\infty$-categories.
    Finally, the right adjoints are given by taking hearts.
    For details, see~\cite[C.5.4.5, C.5.9.3]{sag}.

    The failure of the displayed left adjoint functors to preserve limits is
    behind the proliferation of derived categories attached to a single scheme
    $X$. Starting with $\QCoh(X)$, we can go to the left or right to obtain
    $\Dscr(\QCoh(X))_{\geq 0}$ and $\wDscr(\QCoh(X))_{\geq 0}$. We know that
    $\Dscr_\qc(X)_{\geq 0}$ is complete and separated. If it is weakly
    $0$-complicial, then this implies that it is equivalent to
    $\wDscr(\QCoh(X))_{\geq 0}$. If it is $0$-complicial, then it is also
    weakly $0$-complicial and it is equivalent to
    $\Dscr(\QCoh(X))_{\geq 0}$ and to $\wDscr(\QCoh(X))_{\geq 0}$. Thus, we see that
    $\Dscr_\qc(X)_{\geq 0}$ is $0$-complicial if and only if $$\Dscr(\QCoh(X))_{\geq
    0}\we\wDscr(\QCoh(X))_{\geq 0}\we\Dscr_\qc(X)_{\geq 0}.$$ The next example,
    of Verdier, shows that this does not always happen.
\end{remark}

\begin{example}\label{ex:not0complicial}
    Let $Z$ be the Verdier example~\cite[Expos\'e~II, Appendix I]{sga6} of a quasi-compact quasi-separated scheme
    obtained by gluing two copies of a specific affine scheme $X=\Spec R$ together
    along a specific quasi-compact open. Then, $\Dscr(\QCoh(Z))_{\geq 0}$ is not
    equivalent to $\Dscr_{\qc}(Z)_{\geq 0}$. So, we see that
    $\Dscr_{\qc}(Z)_{\geq 0}$ is not $0$-complicial.
\end{example}

%
%

\subsection{The singularity category}\label{sub:sing}

Here we discuss the connection between the unseparated and separated derived
categories and singularity categories.

Consider a general Grothendieck abelian category. There is a localization
sequence
$$\Kscr(\mathrm{AcInj}_\Ascr)\xrightarrow{I}\cDscr(\Ascr)\xrightarrow{Q}\Dscr(\Ascr),$$
where $\Kscr(\mathrm{AcInj}_\Ascr)$ is obtained as the dg nerve of the full dg
subcategory of $\mathrm{Inj}_\Ascr$ on those unbounded complexes of injectives
which are additionally acyclic (quasi-isomorphic to zero). In particular, $I$
and $Q$ admit {\em right} adjoints $I_\rho$ and $Q_\rho$, respectively.
See~\cite[Proposition~3.6]{krause-noetherian}; this also follows from the
localization theory of~\cite[Section~C.5.2]{sag}.

When $\Ascr$ is locally noetherian and $\Dscr(\Ascr)$ is compactly generated,
Krause shows in~\cite[Theorem~1.1]{krause-noetherian} that
$I$ and $Q$ admit additional {\em left}
adjoints $I_\lambda$ and $Q_\lambda$, forming a r\'ecollement
$$\Kscr(\mathrm{AcInj}_\Ascr)\lrlarrows\cDscr(\Ascr)\lrlarrows\Dscr(\Ascr).$$
In particular, since $Q$ preserves colimits, we see that $Q_\lambda$ preserves
compact objects, that $\Kscr(\mathrm{AcInj}_\Ascr)$ is compactly generated, and
that there is an exact sequence
$$\Dscr(\Ascr)^\omega\rightarrow\cDscr(\Ascr)^\omega\rightarrow\Kscr(\mathrm{AcInj}_\Ascr)^\omega$$
of small idempotent complete stable $\infty$-categories. In this setting,
$\cDscr(\Ascr)^\omega\we\Dscr^b(\Ascr^\omega)$.

\begin{example}
    If $X$ is a noetherian scheme with affine diagonal, then this gives the
    familiar exact sequence
    $$\Perfscr(X)\rightarrow\Dscr^b(\Coh(X))\rightarrow\Dscr_\sing(X),$$
    where $\Dscr_\sing(X)$ denotes the natural $\infty$-categorical enhancement
    of the {\bf singularity category} of $X$.
\end{example}

We have seen in this paper that $\Perfscr(X)$ and $\Dscr^b(X)$ both admit
unique $\infty$-categorical enhancements when $X$ is noetherian with affine
diagonal. It is natural to ask about $\Dscr_\sing(X)$.

\begin{example}\label{ex:sds}
    The work of Schlichting~\cite{schlichting} and
    Dugger--Shipley~\cite{dugger-shipley-curious} (in the $p>3$ case) and
    Muro--Raptis~\cite{muro-raptis} (in the $p=2,3$ case)
    shows that $\D_\sing(\ZZ/p^2)\we\D_\sing(\FF_p[\epsilon]/(\epsilon^2))$
    while $\Dscr_\sing(\ZZ/p^2)$ is not equivalent to
    $\Dscr_\sing(\FF_p[\epsilon]/(\epsilon^2))$. Thus, we see that even in the
    best possible case, where $\Perfscr(X)$ and $\Dscr^b(X)$ admit unique
    $\infty$-categorical enhancements, the singularity category can admit
    non-unique enhancements.
\end{example}

\begin{remark}
    The Schlichting and Dugger--Shipley work also implies that the large
    triangulated category $\K(\mathrm{AcInj}_\Ascr)$ admits multiple
    non-equivalent
    presentable $\infty$-categorical enhancements, when $\Ascr=\Mod_{\ZZ/p^2}$.
    Indeed, it is equivalent to the homotopy category of
    $\Kscr(\mathrm{AcInj}_\Bscr)$, where
    $\Bscr=\Mod_{\FF_p[\epsilon]/(\epsilon)^2}$, but
    $\Kscr(\mathrm{AcInj}_\Bscr)$ is not equivalent to
    $\Kscr(\mathrm{AcInj}_\Ascr)$.
\end{remark}

\begin{remark}
    Passing to the singularity category (either its big or small version)
    destroys the presence of $t$-structures, which is why our methods (or those
    of~\cite{lunts-orlov,cs-uniqueness}) do not apply.
    As an example, let $R=\FF_p[C_p]$, the group ring of the cyclic group of
    order $p$ over $\FF_p$. The singularity category $\Dscr_\sing(R)$ is
    generated by the image of the trivial $C_p$-module $\FF_p$. The
    endomorphism ring of $\FF_p$ in $\Dscr_\sing(R)$ computes the Tate-cohomology $\widehat{\H}^*(C_p,\FF_p)$,
    which is $2$-periodic. In particular, there cannot be a left and right
    separated $t$-structure on $\D_\sing(R)$.
\end{remark}

\subsection{Stable $n$-categories}\label{sub:nstable}

In this section we attempt to outline a story which will eventually clarify the
power of triangulated categories at determining the underlying stable
$\infty$-category, despite losing a great deal of information.
Most of this section is speculative. But, Raptis has recently made some
progress in the directions outlined here (and much more) in~\cite{raptis}.

If $n\geq 1$, an $n$-category for us is an $\infty$-category $\Cscr$ such that
$\Map_\Cscr(X,Y)$ is $(n-1)$-truncated for all objects $X,Y\in\Cscr$. Recall
that this means that $\pi_i\Map_\Cscr(X,Y)=0$ for all $i\geq n$ (and every
choice of basepoint). We let
$\Cat_{n-1}\subseteq\Cat_\infty$ be the full subcategory on
the $n$-categories. In general, $\Cat_{n-1}$ is
itself a large $n$-category. In particular, $\Cat_0$ is
equivalent to the category of small categories and functors between them. By
Gepner--Haugseng~\cite{gepner-haugseng}, we
can also view $\Cat_{n-1}$ as an $(n,2)$-category as $\Cat_{n-1}$ is enriched
over itself: given $n$-categories $\Cscr$ and $\Dscr$, the functor
$\infty$-category $\Fun(\Cscr,\Dscr)$ is an $n$-category by~\cite[2.3.4.8]{htt}.

\begin{remark}
    The indexing comes from higher topos theory. Given an $\infty$-topos
    $\Cscr$, the full subcategory $\Cscr_{\leq 0}$ of $0$-truncated objects is
    a topos. More generally, the full subcategory $\Cscr_{\leq n-1}$ is an
    $n$-topos.
\end{remark}

The inclusion
$\Cat_{n-1}\subseteq\Cat_\infty$ admits a left adjoint
$h_{n-1}$. Write $h_{n-1}\Cscr$ for the {\bf $n$-homotopy category} of an
$\infty$-category $\Cscr$. The $n$-category $h_{n-1}\Cscr$ has the same
objects as $\Cscr$, but $\Map_{h_{n-1}\Cscr}(X,Y)\we\tau_{\leq
n-1}\Map_\Cscr(X,Y)$. For details, see~\cite[Section~2.3.4]{htt}.

\begin{conjecture}\label{conj:nstable}
    For $1\leq n\leq \infty$, there exists a good theory of {\bf stable
    $n$-categories} and exact functors
    between them. This theory should fit into the following picture.
    \begin{enumerate}
        \item[{\rm (i)}] Stable $n$-categories and exact functors form an
            $(n,2)$-category $\Cat_{n-1}^\ex$ which is equipped with a
            forgetful functor
            $u_{n-1}\colon\Cat_{n-1}^\ex\rightarrow\Cat_{n-1}$.
            In particular, given stable $n$-categories $\Cscr$ and $\Dscr$,
            there should be an $n$-category $\Fun^\ex(\Cscr,\Dscr)$ of exact functors.
        \item[{\rm (ii)}] For $n\geq k$, there is a $k$-homotopy category
            functor $h_{k-1}\colon\Cat_{n-1}^\ex\rightarrow\Cat_{k-1}^\ex$ which fits into a
            commutative square
            $$\xymatrix{
            \Cat_{n-1}^\ex\ar[r]^{h_{k-1}}\ar[d]^{u_{n-1}}&\Cat_{k-1}^\ex\ar[d]^{u_{k-1}}\\
            \Cat_{n-1}\ar[r]^{h_{k-1}}&\Cat_{k-1}}$$
            of $n$-categories.
        \item[{\rm (iii)}] The $(\infty,2)$-category $\Cat_\infty^\ex$ is
            equivalent to the usual $\infty$-category of
            stable $\infty$-categories, exact functors, and natural
            transformations.
        \item[{\rm (iv)}] The $(1,2)$-category $\Cat_0^\ex$ is equivalent to
            the category of triangulated categories, exact functors, and
            natural transformations.
    \end{enumerate}
\end{conjecture}

\begin{remark}
    In particular, if $\Cscr$ is a stable $\infty$-category, then
    $h_{n-1}\Cscr$ is a stable $n$-category. As a special case, $h_0\Cscr$ is
    the usual triangulated homotopy category of $\Cscr$.
\end{remark}

\begin{remark}
    Given a stable $\infty$-category $\Cscr$, the suspension functor
    $\Sigma\colon\Cscr\rightarrow\Cscr$ induces an automorphism
    $\Sigma\colon h_{n-1}\Cscr\rightarrow h_{n-1}\Cscr$ of each associated
    stable $n$-category. Thus, a stable $n$-category should in particular be an
    $n$-category equipped with a fixed automorphism and exact functors should
    preserve these.
\end{remark}

\begin{remark}
    Fix $k\geq 1$ an integer and $p$ a prime.
    Motivated on the work of
    Barthel--Schlank--Stapleton~\cite{barthel-schlank-stapleton} on the asymptotic
    algebraicity of chromatic homotopy theory and of Patchkoria~\cite{patchkoria} on exotic equivalences, Piotr
    Pstr{\k{a}}gowski~\cite{pstragowski} has
    recently proved the remarkable theorem that if $E$ is a $p$-local Landweber
    exact homology theory of height $n$ such that $p>n^2+n+1+\tfrac{k}{2}$, then the
    stable $k$-homotopy categories $h_{k-1}\Sp_E$ and $h_{k-1}\Dscr(E_*E)$ are
    equivalent, where $\Sp_E$ denotes the $E$-local stable homotopy category
    and $\Dscr(E_*E)$ is the derived $\infty$-category of $E_*E$-comodules.
\end{remark}

We make no attempt here to prove this conjecture. However, we note that it
explains certain phenomena.

\begin{example}
    Consider the Schlichting and Dugger--Shipley examples as in~\ref{ex:sds}.
    The proof given in Dugger--Shipley that $\Dscr_\sing(\ZZ/p^2)$ is not
    equivalent to $\Dscr_\sing(\FF_p[\epsilon]/(\epsilon^2))$ works as follows.
    The class of $\FF_p$ itself (where either $p=0$ or $\epsilon=0$) generates
    the singularity category. Write $A$ for the endomorphism ring spectrum of
    $\FF_p$ in $\Dscr_\sing(\ZZ/p^2)$ and write $A_\epsilon$ for the
    endomorphism ring of $\FF_p$ in
    $\Dscr_\sing(\FF_p[\epsilon]/(\epsilon^2))$. The homotopy rings $\pi_*A$
    and $\pi_*A_\epsilon$ are isomorphic, but the connective covers $\tau_{\geq
    0}A$ and $\tau_{\geq 0}A_\epsilon$ are not equivalent (so that $A$ and
    $A_\epsilon$ are not equivalent). Let $B=\tau_{\leq 2}\tau_{\geq 0}A$
    and $B_\epsilon=\tau_{\leq 2}\tau_{\geq 0}A_\epsilon$. Dugger and
    Shipley show in fact that $B$ is not equivalent to $B_\epsilon$. What
    this means is that the stable $3$-categories $h_2\Dscr_\sing(\ZZ/p^2)$ and
    $h_2\Dscr_\sing(\FF_p[\epsilon]/(\epsilon^2))$ are not equivalent.
\end{example}

Schlichting proved that the algebraic $K$-theory of $A$ and $A_\epsilon$
differ. This motivates the following conjecture, which Schlichting effectively
established for the $n=1$ case of triangulated categories.

\begin{conjecture}\label{conj:non}
    There is no natural number $n$ such that for all small stable $\infty$-categories
    $\Cscr$ and $\Dscr$, if $h_{n-1}\Cscr\we h_{n-1}\Dscr$
    as stable $n$-categories,
    then $\K(\Cscr)\we\K(\Dscr)$, where $\K$ denotes now nonconnective
    algebraic $K$-theory as in~\cite{bgt1}.
\end{conjecture}

\begin{remark}
    This conjecture is true. The following argument was explained to me by Niko
    Naumann and Irakli Patchkoria.
    Fix some $k$, $n$, and $p$ such that
    $p > n^2 + n +1 + \tfrac{k}{2}$. Then, $h_{k-1}\Sp_E \simeq h_{k-1}\Dscr(E_*E)$ where $E$
    is a $p$-local
    Landweber exact homology theory of height $n$. We have
    $L_{K(n)}\K(\text{compact objects in $\Sp_E$})$ is non-zero
    because we can apply the $E$-homology functor $$\text{compact objects in
    $\Sp_E$} \rightarrow \text{compact $E$-modules}$$ to get
    $\K(\text{compact objects in $\Sp_E$}) \rightarrow\K(\text{compact
    $E$-modules})\we\K(E)$ and then compose with the trace
    $\K(E)\rightarrow\THH(E)$ and then cap off the circle to get $\THH(E)
    \rightarrow E$. Here, we use that $L_{K(n)}E$ is non-zero. On the other
    hand, $L_{K(n)}\K(E_*E)=0$ for $n>1$ by Mitchell's theorem~\cite{mitchell}.
\end{remark}

The next conjecture is true for $n=1$ and $i=0$. We are not sure at the moment
what happens for $n=1$ in negative degrees.

\begin{conjecture}
    Let $\Cscr$ and $\Dscr$ be small stable $\infty$-categories.
    If $h_{n-1}\Cscr\we h_{n-1}\Dscr$ as stable $n$-categories, then
    $\K_i(\Cscr)\iso\K_i(\Dscr)$ for $i\leq n-1$. In fact, in this case, we
    guess that $\tau_{\leq n-1}\K(\Cscr)\we\tau_{\leq n-1}\K(\Dscr)$ as spectra.
\end{conjecture}

\begin{remark}
    This conjecture is also true. See~\cite[Theorem~6.11]{raptis} of Raptis.
\end{remark}

Now, we turn to stable $n$-categories and uniqueness of enhancements.
Assuming the theory exists, we can make sense of a $t$-structure on a stable
$n$-category. In the case of $h_{n-1}\Cscr$ where $\Cscr$ is a stable
$\infty$-category, giving a $t$-structure on $\Cscr$ is equivalent to giving
one on $h_{n-1}\Cscr$. For $n=1$, this was Lemma~\ref{lem:t}.
Given a $t$-structure on $h_{n-1}\Cscr$, the heart is still an abelian
category. The next definition is due to Lurie~\cite[Section~C.5.4]{sag}.

\begin{definition}
    A Grothendieck abelian $n$-category is an $n$-category equivalent to
    $\tau_{\leq n-1}\Cscr$ for a Grothendieck prestable $\infty$-category
    $\Cscr$. We let $\mathrm{Groth}_{n-1}\subseteq\Pr^\L$ be the full subcategory of
    presentable $\infty$-categories on the Grothendieck abelian $n$-categories.
\end{definition}

\begin{example}
    The full subcategory $\Dscr(\ZZ)_{[0,n-1]}\subseteq\Dscr(\ZZ)$ of complexes $X$ with
    $\H_i(X)=0$ for $i\notin[0,n-1]$ is a Grothendieck abelian $n$-category.
\end{example}

The next proposition relates $n$-complicial Grothendieck prestable $\infty$-categories to Grothendieck abelian
$n$-categories and stable $n$-categories.

\begin{proposition}
    Suppose that $\Ascr$ and $\Bscr$ are $n$-complicial separated Grothendieck prestable
    $\infty$-categories. Assuming Conjecture~\ref{conj:nstable}, the following conditions are equivalent:
    \begin{enumerate}
        \item[{\rm (a)}] the stable $\infty$-categories $\Sp(\Ascr)$ and
            $\Sp(\Bscr)$ are equivalent;
        \item[{\rm (b)}] the abelian $n$-categories $\Ascr_{\leq n-1}$ and
            $\Bscr_{\leq n-1}$ are equivalent;
        \item[{\rm (c)}] the stable $n$-categories $h_{n-1}\Sp(\Ascr)$ and
            $h_{n-1}\Sp(\Bscr)$ are equivalent.
    \end{enumerate}
\end{proposition}

\begin{proof}
    The equivalence of (a) and (b) is proved in~\cite[C.5.4.5]{sag}. Clearly, (a)
    implies (c). So, assume (c). Then, $$\Ascr_{\leq n-1}\we
    (h_{n-1}\Sp(\Ascr))_{[0,n-1]}\we(h_{n-1}\Sp(\Bscr))_{[0,n-1]}\we\Bscr_{\leq
    n-1},$$ which is exactly (b). Here, $(h_{n-1}\Sp(\Ascr))_{[0,n-1]}$ refers
    to the objects in the given range in the $t$-structure on the stable
    $n$-category $h_{n-1}\Sp(\Ascr)$.
\end{proof}

The $n=0$ case of the next conjecture is exactly our Theorem~\ref{mt:2}.

\begin{conjecture}
    Let $\Ascr$ be an $n$-complicial separated Grothendieck prestable
    $\infty$-category. Suppose that $\Cscr$ is a stable presentable
    $\infty$-category together with an exact equivalence $h_{n-1}\Cscr\we
    h_{n-1}\Sp(\Ascr)$ of stable $n$-categories. Then, $\Cscr\we\Sp(\Ascr)$.
\end{conjecture}

\begin{remark}
    One can sketch a proof along the lines of our proof of Theorem~\ref{mt:2}.
    However, it obviously depends on the notion of an exact functor of stable
    $n$-categories, so it will not be rigorous at the moment.
\end{remark}

In terms of the dg enhancement of this kind of question, we simply give the
following example.

\begin{example}\label{ex:dugger-shipley}
    Dugger and Shipley~\cite{dugger-shipley-topological} give dg algebras $A$ and $B$ over $\ZZ$ with
    $\pi_*A\iso\pi_*B\iso\Lambda_{\FF_2}(g_2)$, where $|g_2|=2$. They are equivalent as
    $\SS$-algebras but not as as $\ZZ$-algebras. In fact, they are not even Morita
    equivalent over $\ZZ$. Thus, $\D(A)\we\D(B)$ admits two distinct dg categorical
    enhancements. Moreover, those enhancements are separated and $2$-complicial. Hence, we
    see that the Grothendieck abelian $3$-categories $\Dscr(A)_{[0,2]}$ and $\Dscr(B)_{[0,2]}$ are
    different as $\Dscr(\ZZ)_{[0,2]}$-linear $3$-categories;
    equivalently, the stable $3$-categories $h_2\Dscr(A)$ and $h_2\Dscr(B)$ are not
    equivalent as $h_2\Dscr(\ZZ)$-linear stable $3$-categories.
\end{example}

\subsection{Enhancements and $t$-structures}\label{sub:conjecture}

Known examples of triangulated categories not admitting enhancements or
admitting multiple enhancements do not admit obvious finitely-complicial $t$-structures.
Situations in which there is a bounded $t$-structure seem to be much closer to
algebra and we conjecture that they exhibit a strong rigidity property with
respect to their enhancements.

We give two wildly optimistic conjectures in this section.

\begin{conjecture}\label{conj:crazy}
    Let $\Cscr$ be a stable presentable $\infty$-category with an accessible
    right complete $t$-structure which is
    compatible with filtered colimits and is additionally $n$-complicial for some
    $n$. Then, the homotopy category $h\Cscr$ admits a unique $\infty$-categorical enhancement.
\end{conjecture}

The reason it is not too crazy to ask for the conjecture to be true is because of the
work~\cite{schwede-unique,schwede-shipley,schwede-rigid} of Schwede--Shipley
and Schwede on the homotopy category of spectra.
For example, Schwede proves that the homotopy category $h\Sp$ admits a unique
stable model category enhancement. The argument is basically to study Toda
brackets, which can be constructed using only the triangulated structure, and
to appeal to the fact that the stable homotopy ring (at each prime) can be generated by Toda
brackets of certain low-degree classes in a sense made precise
in~\cite{schwede-unique}.\footnote{
The argument of~\cite{schwede-rigid} can be re-written to
prove that $h\Sp$ admits a unique $\infty$-categorical enhancement in the sense
of this paper. Note also that the argument involving Toda brackets takes place entirely within the
prestable $\infty$-category $\Sp_{\geq 0}$ of connective spectra.}

We thank S. Schwede for bringing to our attention the following evidence for
Conjecture~\ref{conj:crazy}.

\begin{example}
    The unpublished thesis of K. Hutschenreuter~\cite{hutschenreuter} establishes the
    conjecture for $\Cscr\we\Dscr(\tau_{\leq n}\SS_{(p)})$ for $n\geq
    p^2(2p-2)-1$ when $p$ is an odd prime and for $n\geq 0$ when $p=2$.
\end{example}

\begin{example}
    Consider $R=\tau_{\leq 2}\SS$, which is an $\EE_\infty$-ring spectrum with
    non-zero homotopy groups
    $\pi_0R\iso\ZZ$ and $\pi_1R\iso\pi_2R\iso\ZZ/2$. In fact,
    $\pi_*R\iso\ZZ[\eta]/(2\eta,\eta^3)$, where $|\eta|=1$. As can be found
    in~\cite[p. 177]{mosher-tangora}, we have a Toda bracket $\eta^2\in\langle
    2,\eta,2\rangle$. In particular $\D(R)$ is not equivalent to
    $\D(\pi_*R)$. The Toda brackets somehow help us capture certain higher
    homotopical bits of information in the homotopy category.
\end{example}

The condition that $\Cscr$ be $n$-complicial for some $n$ is critical.

\begin{example}
    Fix a prime $p\geq 5$ and 
    consider a Brown--Peterson ring spectrum $\BP\langle 1\rangle$, which is a
    certain connective $\EE_1$-ring spectrum with $\pi_*\BP\langle 1\rangle\iso\ZZ_{(p)}[v_1]$ where $|v_1|=2p-2$.
    Patchkoria has proven in~\cite[Theorem~1.1.3]{patchkoria} that there is a triangulated equivalence $h\Dscr(\BP\langle 1\rangle)\we
    h\Dscr(\pi_*\BP\langle 1\rangle)$, where $\pi_*\BP\langle 1\rangle$ is
    viewed as a formal $\EE_1$-ring spectrum. This equivalence cannot come from an
    equivalence of the underlying stable $\infty$-categories, so we have
    another example of triangulated categories with multiple
    $\infty$-categorical enhancements. Both sides admit an accessible right and
    left-complete $t$-structure each of which is compatible with filtered colimits;
    in each case the heart is $\Mod_{\ZZ_{(p)}}$. However, since $\BP\langle 1\rangle$ is
    not bounded above, the $t$-structure on $\Dscr(\BP\langle 1\rangle)$ is not
    $n$-complicial for any $n$ by~\cite[C.5.5.15]{sag}.
\end{example}

We can use this example to prove the following theorem.

\begin{theorem}\label{thm:nogo}
    There exists a small triangulated category $\T$ with a bounded
    $t$-structure which admits multiple non-equivalent enhancements.
\end{theorem}

\begin{proof}
    Fix $p\geq 5$ and $\BP\langle 1\rangle$ as above.
    Let $\Dscr^b(\BP\langle 1\rangle)\subseteq\Dscr(\BP\langle
    1\rangle)^\omega$ be the full subcategory of $v_1$-torsion objects and
    similarly for $\Dscr^b(\pi_*\BP\langle 1\rangle)$. Equivalently, these are
    the subcategories consisting of bounded objects with finitely presented
    homotopy groups. We have an equivalence $h\Dscr^b(\BP\langle 1\rangle)\we
    h\Dscr^b(\pi_*\BP\langle 1\rangle)$, but this cannot lift to an equivalence
    of stable $\infty$-categories. Indeed, $\BP\langle 1\rangle/v_1^k$ does
    not admit the structure of an Eilenberg--Mac Lane spectrum for $k\geq 2$.
    See for example~\cite[Remark~5.4]{dugger-shipley-topological} which shows
    that the reduction of $\BP\langle 1\rangle/v_1^k$ modulo $p$ does not admit
    the structure of a $\ZZ$-algebra in spectra.
\end{proof}

The $t$-structures in Theorem~\ref{thm:nogo} are not $n$-complicial for any
$n$. This suggests the following conjecture.

\begin{conjecture}
    Let $\T$ be a small triangulated category with a bounded $t$-structure
    which is $n$-complicial for some $n$. Then, $\T$ admits a unique
    $\infty$-categorical which is unique.
\end{conjecture}

We do not conjecture that all such admit unique dg enhancements. Indeed, this
can already be seen to be false by looking at examples cooked up from different
dg $\ZZ$-algebra structures on the same $\EE_1$-ring spectrum.

\begin{example}
    Consider the dg $\ZZ$-algebras $A$ and $B$ of
    Example~\ref{ex:dugger-shipley}. These are both noetherian $\EE_1$-rings
    and we can thus consider $\Dscr^b(A)$, the full subcategory of $\Dscr(A)$
    on the bounded objects $X$ with $\pi_nX$ finitely generated over $\pi_0A$
    for all $n$. Make a similar definition for $\Dscr^b(B)$. The equivalence
    $\D(A)\we\D(B)$ preserves the triangulated subcategories
    $\D^b(A)=h\Dscr^b(A)$ and $\D^b(B)=h\Dscr^b(B)$. But, $\Dscr^b(A)$ is not
    equivalent to $\Dscr^b(B)$.
\end{example}

\subsection{Category theory questions}\label{sub:categories}

As far as we know, the next question could have a positive answer in all cases.
If so, it would allow us to remove presentability from Theorem~\ref{mt:1} in a
more simple way than we do in Appendix~\ref{sub:appendix}.

\begin{question}
    Let $\Cscr$ and $\Dscr$ be stable $\infty$-categories with a triangulated
    equivalence $\Ho(\Cscr)\we\Ho(\Dscr)$. If $\Cscr$ is presentable, is
    $\Dscr$ presentable? What if $\Cscr$ admits additionally an accessible
    $t$-structure which is compatible with filtered colimits?
\end{question}

Here is a related question.

\begin{question}
    Suppose that $\Cscr$ and $\Dscr$ are small stable $\infty$-categories with
    $h\Cscr\we h\Dscr$. Is it true that $h\Ind(\Cscr)\we h\Ind(\Dscr)$?
    Certainly this is the case if in fact $\Cscr\we\Dscr$.
\end{question}

\appendix
\section{Appendix: removing presentability}\label{sub:appendix}

Theorem~\ref{mt:2} is the $\infty$-categorical analogue of a theorem of
Canonaco and Stellari in the case of dg enhancements. Their
theorem notably does not include a presentability hypothesis. We indicate how to use
their approach, which itself builds on the work of
Lunts--Orlov~\cite{lunts-orlov}, to remove the presentability qualification in
Theorem~\ref{mt:2}.

True to the spirit of this paper, we use prestable $\infty$-categories in the
proof. This replaces the use of the brutal truncations in Sections 3 and 4
of~\cite{cs-uniqueness} and makes it somewhat easier to establish the existence
of the comparison functor $F'$ in the proof.

\begin{theorem}\label{thm:appendix}
    If $\Ascr$ is a Grothendieck abelian category, then $\D(\Ascr)$ admits a
    unique $\infty$-categorical enhancement.
\end{theorem}

\begin{proof}
    Let $\Cscr$ be an $\infty$-categorical enhancement of $\D(\Ascr)$ and let
    $F\colon h\Cscr\we\D(\Ascr)$ be a fixed equivalence. Then,
    the Postnikov $t$-structure on $\D(\Ascr)$ induces a $t$-structure on
    $\Cscr$ which is right separated and compatible with countable coproducts.
    In particular, it is right complete by Proposition~\ref{prop:sepcomplete}. It follows,
    by Lemma~\ref{lem:sp}, that to prove $\Cscr\we\Dscr(\Ascr)$ it is enough to prove
    that $\Cscr_{\geq 0}\we\Dscr(\Ascr)_{\geq 0}$. Choose a generator $X$ of
    the abelian category $\Ascr$ and let $R=\Hom_\Ascr(X,X)$ be the ring of
    endomorphisms of $X$. Since $\Dscr(\Ascr)_{\geq 0}$ is
    $0$-complicial, it follows by the $\infty$-categorical Gabriel--Popescu
    theorem~\cite[C.2.1.6]{sag} that the natural functor $\Dscr(\Ascr)_{\geq
    0}\rightarrow\Dscr(R)_{\geq 0}$ is fully faithful and admits a left exact
    left adjoint. Let $\Kscr$ be the kernel of this functor.
    By~\cite[C.2.3.8]{sag},
    $\Kscr$ is itself a Grothendieck prestable $\infty$-category and
    $\Kscr^\heart\subseteq\Mod_R$ is a full Serre subcategory. Since
    $\Dscr(\Ascr)_{\geq 0}$ is separated, by~\cite[C.5.2.4]{sag} we see that $\Kscr$ is the full subcategory of
    $\Dscr(R)_{\geq 0}$ consisting of those objects $Y$ such that
    $\pi_iY\in\Kscr^\heart$ for all $i$.

    Let $\Ascr_0\subseteq\Ascr$ be the full subcategory consisting of finite
    direct sums of the object $X$. Then, $\Ascr_0$ is an additive
    $\infty$-category and $\Pscr_\Sigma(\Ascr_0)\we\Dscr(R)_{\geq 0}$. Here,
    $\Pscr_\Sigma(\Ascr_0)\subseteq\Pscr(\Ascr_0)$, called the {\bfseries
    nonabelian derived category}, is the full subcategory of
    functors $\Ascr_0^\op\rightarrow\Sscr$ which preserve finite products.
    In particular, by~\cite[5.5.8.15]{htt}, to give a colimit preserving functor
    $\Pscr_\Sigma(\Ascr_0)\rightarrow\Cscr_{\geq 0}$ is the same as giving a
    finite coproduct
    preserving functor $\Ascr_0\rightarrow\Cscr_{\geq 0}$. Such a functor is canonically
    induced by $F$. Thus, we have a diagram of left adjoint functors
    $$\xymatrix{&\Dscr(R)_{\geq 0}\ar[rd]^{P'}\ar[ld]_L&\\\Dscr(\Ascr)_{\geq
    0}&&\Cscr_{\geq 0}.}$$ We first show that $P'$ factors through $L$, or in other words
    that there exists a functor $F'\colon\Dscr(\Ascr)_{\geq
    0}\rightarrow\Cscr_{\geq 0}$ and an equivalence of functors $P'\we F'\circ
    L$. If such a factorization exists, it is unique since $L$ is a
    localization.

    To prove the existence of the factorization, it is enough to prove that if $Y\in\Kscr$, then $P'(Y)\we 0$.
    If $Y$ is bounded, then this follows immediately from the fact that $P'$ and
    $L$ are compatible with the equivalence
    $F^\heart\colon\Cscr^\heart\we\Ascr$. In fact, more generally, we see by
    Proposition~\ref{prop:ba} that $\Cscr^-\we\Dscr^-(\Ascr)$ and that this
    identification is compatible with $P'$ and $L$. Because filtered colimits in
    $\Cscr_{\geq 0}$ are left exact using Lemma~\ref{lem:detection}, a careful reading of the proofs
    of~\cite[C.2.5.2 and C.3.2.1]{sag} implies that the statement
    of~\cite[C.2.5.2]{sag} applies to $\Cscr_{\geq 0}$ even though we do not
    yet know that it is Grothendieck prestable (but we do know that it is
    prestable and has all limits and colimits). In particular, $P'$ is left
    exact. Suppose now that $Y$ is a general object of $\Kscr$. Then,
    $\pi_iY\in\Kscr^\heart$ for each $i$ and hence each truncation
    $\tau_{\leq n}Y$ is also in $\Kscr$. By the observation about bounded
    objects made above, we know that $P'(\tau_{\leq n}Y)\we
    0$ for all $n$. In particular, since $P'$ is left exact, $P'(\tau_{\geq n+1}Y)\we P'(Y)$ for all $n$.
    Hence, $P'(Y)$ is $\infty$-connective and therefore $P'(Y)\we 0$ as $\Cscr$
    is left separated. Thus, the
    factorization exists as claimed, using~\cite[C.2.3.10]{sag} (which
    extends to the case where the target category is merely cocomplete
    prestable and not necessarily Grothendieck prestable).

    We have shown the existence of a left exact functor $F'\colon\Dscr(\Ascr)_{\geq
    0}\rightarrow\Cscr_{\geq 0}$ which preserves colimits. Let $U$ be the
    right adjoint of $L$. We now want to prove that $F'$ is fully faithful. It is enough to prove
    that $\Map_{\Dscr(\Ascr)}(X,Y)\rightarrow\Map_\Cscr(F'X,F'Y)$ is an equivalence
    for all $Y\in\Dscr(\Ascr)_{\geq 0}$. For this, it is enough to show that
    $\Hom_{\D(\Ascr)}(X,Y)\rightarrow\Hom_{h\Cscr}(F'X,F'Y)$ is an isomorphism
    for all $Y\in\Dscr(\Ascr)_{\geq 0}$.

    Write $P=F\circ hL\colon\D(R)_{\geq 0}\rightarrow h\Cscr_{\geq 0}$. 
    We obtain a diagram
    $$\Hom_{h\Cscr}(PR,PUY)\leftarrow\Hom_{\D(R)}(R,UY)\iso\Hom_{\D(\Ascr)}(X,Y)\rightarrow\Hom_{h\Cscr}(P'R,P'UY);$$
    by adjunction the left arrow is an isomorphism. Fully faithfulness now
    follows from Lemma~\ref{lem:lo} below. We then have that
    $F'\colon\Dscr(\Ascr)\rightarrow\Cscr$ is fully faithful. Since every
    object of $h\Cscr$ is of the form $FY$ for some $Y\in\Dscr(\Ascr)$ and
    since $FY\we PUY\we P'UY\we F'Y$ (also by the next lemma), we see that $F'$ is essentially surjective.
    Thus, $F'$ is an equivalence and we are done.
\end{proof}

The next lemma is our version of~\cite[Proposition~3.5]{cs-uniqueness}, which
is itself a generalization of~\cite[Proposition~3.4]{lunts-orlov}. Our proof
closely follows the arguments of Lunts--Orlov and Canonaco--Stellari with
easy adjustments for the $\infty$-categorical and prestable setting.

\begin{lemma}\label{lem:lo}
    Let $\Cscr_{\geq 0}$ be a prestable $\infty$-category with all limits and
    colimits and let
    $P_0,P_1\colon\D(R)_{\geq 0}\rightrightarrows h\Cscr_{\geq 0}$ be two
    coproduct preserving exact functors. Let
    $\Ascr_0=\Mod_R^{\proj,\omega}\subseteq\D(R)_{\geq 0}$ be the full subcategory of
    finitely generated projective left $A$-modules. Suppose that there is a
    natural isomorphism $\theta_0\colon P_0|_{\Ascr_0}\iso P_1|_{\Ascr_0}$ of
    the functors $P_0$ and $P_1$ restricted to $\Ascr_0$. For each object $Y\in\D(R)_{\geq
    0}$, there is an isomorphism $\sigma_Y\colon P_0Y\rightarrow P_1Y$ such that
    for each $f\colon R\rightarrow Y$, the diagram
    $$\xymatrix{P_0R\ar[r]^{P_0(f)}\ar[d]^{\theta_0}&P_0Y\ar[d]^{\sigma_Y}\\
    P_1R\ar[r]^{P_1(f)}&P_1Y}$$ commutes in $h\Cscr_{\geq 0}$.
\end{lemma}

Since $\Cscr_{\geq 0}$ is a full subcategory of the stable $\infty$-category
$\SWscr(\Cscr_{\geq 0})$, the homotopy category $h\Cscr_{\geq 0}$ is a full subcategory of
the triangulated category $h\SWscr(\Cscr_{\geq 0})$. Moreover, $h\Cscr_{\geq 0}$ is closed
under cones, coproducts, and positive shifts in $h\SWscr(\Cscr_{\geq 0})$. We will use these
facts implicitly in the proof.

\begin{proof}
    We use the nonabelian derived category $\Pscr_\Sigma(\Ascr_0)$ which
    appeared in the proofs of Theorem~\ref{mt:3} and Theorem~\ref{thm:appendix}. Each object $Y$ of
    $\Pscr_\Sigma(\Ascr_0)$ can be represented by a simplicial object
    $Y_\bullet\colon\Delta^\op\rightarrow\Ind(\Ascr_0)$, where $\Ind(\Ascr_0)$ is an additive
    category with filtered colimits, but typically not all colimits. We can
    also assume that each $Y_n$ is a projective object of $\Mod_R$.
    Filtering by skeletons, we see that $Y$ admits a filtration $F_\star Y$ where
    $F_i Y\we 0$ for $i<0$ and $\gr^F_iY\in\Mod_R^\proj[i]$ for $i\geq 0$.

    The natural isomorphism $\theta_0$ extends to $P_0$ and $P_1$ when restricted to $\Ind(\Ascr_0)$.
    Fix $n>0$. There are natural isomorphisms $\theta_n$ of $P_0$ and $P_1$
    when restricted to $\Ind(\Ascr_0)[n]\subseteq\Pscr_\Sigma(\Ascr_0)$. We also have
    natural isomorphisms $\theta_i[1]\we\theta_{i+1}$ for $i\geq 0$.

    We will prove inductively that there exist
    isomorphisms $\sigma_i\colon P_0(F_iY)\rightarrow P_1(F_iY)$ and
    $\sigma_Y\colon P_0(Y)\rightarrow P_1(Y)$ such that the diagrams of exact
    sequences
    \begin{equation}\label{33}\xymatrix{
    P_0(F_{i-1}Y)\ar[r]\ar[d]^{\sigma_{i-1}}&P_0(F_iY)\ar[d]^{\sigma_i}\ar[r]&P_0(\gr^F_iY)\ar[d]^{\theta_i}\\
    P_1(F_{i-1}Y)\ar[r]&P_1(F_iY)\ar[r]&P_1(\gr^F_iY)
    }\end{equation} and
    $$\xymatrix{
    \bigoplus_i P_0(F_iY)\ar[r]\ar[d]^{\bigoplus_i\sigma_i}&\bigoplus_i
        P_0(F_iY)\ar[d]^{\bigoplus_i\sigma_i}\ar[r]&P_0(Y)\ar[d]^{\sigma_Y}\\
    \bigoplus_i P_1(F_iY)\ar[r]&\bigoplus_i P_1(F_iY)\ar[r]&P_1(Y)
    }$$
    commute in $h\Cscr_{\geq 0}$.

    Suppose we have inductively chosen
    isomorphisms $\sigma_n\colon P_0(F_iY)\rightarrow
    P_1(F_iY)$ for $0\leq i\leq n$ in $h\Cscr_{\geq 0}$
    such that for each $0\leq i< n$ the diagram
    $$\xymatrix{
        P_0(F_iY)\ar[r]\ar[d]^{\sigma_i}& P_0(F_{i+1}Y)\ar[d]^{\sigma_{i+1}}\ar[r]& P_0(\gr_{i+1}^F
                    Y)\ar[d]^{\theta_{i+1}}\\
            P_1(F_iY)\ar[r]& P_1(F_{i+1}Y)\ar[r]& P_1(\gr_{i+1}^F Y)}$$
    commutes. (The $i=0$ case follows because $\theta_0[1]\we\theta_1$.) We claim that the diagram
    \begin{equation}\label{eq:big}\xymatrix{
        P_0(F_nY)\ar[r]\ar[d]^{\sigma_n}& P_0(F_{n+1}Y)\ar[r]\ar@{.>}[d]&
            P_0(\gr_{n+1}^FY)\ar[d]^{\theta_{n+1}}\ar[r]&P_0(F_nY)[1]\ar[d]^{\sigma_n[1]}\\
        P_1(F_nY)\ar[r]& P_1(F_{n+1}Y)\ar[r]&
        P_1(\gr_{n+1}^FY)\ar[r]&P_1(F_nY)[1]\\
    }\end{equation}
    commutes without the dotted arrow. The left square trivially commutes, since both compositions are zero. 
    Note that by the projectivity of $\gr_{n+1}^FY[-n-1]$, the right hand square itself
    factors naturally into two further squares
    $$\xymatrix{
    P_0(\gr_{n+1}^FY)\ar[d]^{\theta_{n+1}}\ar[r]&
        P_0(\gr_n^FY[1])\ar[r]\ar[d]^{\theta_{n+1}}& P_0(F_nY)[1]\ar[d]^{\sigma_n[1]}\\
    P_1(\gr_{n+1}^FY)\ar[r]&P_1(\gr_n^FY[1])\ar[r]& P_1(F_nY).}$$
    Here, the left hand square again commutes since $\theta_{n+1}$ is a natural
    transformation and the right hand square commutes by our inductive hypothesis and
    the fact that $\theta_{n+1}\we\theta_n[1]$.

    Applying the triangulated category axiom TR3 (see~\cite{neeman} or~\cite{ha}), we see that a dotted map $\sigma_{n+1}$
    exists making diagram~\eqref{eq:big} commute. Thus, by induction, we can choose the
    $\sigma_n$ for all $n$. Recall that the colimit $\colim_n F_nY$ can be computed as the
    cofiber of an appropriate map $\bigoplus F_nY\rightarrow\bigoplus F_nY$.
    Thus, consider the diagram
    $$\xymatrix{
        \bigoplus P_0(F_nY)\ar[r]\ar[d]^{\oplus\sigma_n}&\bigoplus
            P_0(F_nY)\ar[d]^{\oplus\sigma_n}\ar[r]& P_0(Y)\ar[r]\ar@{.>}[d]&\bigoplus
        P_1(F_nY)[1]\ar[d]^{\oplus\sigma_n[1]}\\
            \bigoplus P_1(F_nY)\ar[r]&\bigoplus P_1(F_nY)\ar[r]& P_1(Y)\ar[r]&
            \bigoplus P_1(F_nY)[1]
            }$$
    (where we use that $P_0$ and $P_1$ commute with coproducts).
    This time, the right hand square commutes for trivial reasons. 
    To see that the left hand square commutes, it is enough to check that it commutes when
    restricted to each term $P_0(F_nY)$ of the source. This follows by induction on $F_iY$
    for $0\leq i\leq n$ from the arguments in the second part of the proof. It follows that a dotted arrow exists, which is
    again an isomorphism since the other two arrows are (using the octahedral axiom TR4).

    Now, fix $R\xrightarrow{f} Y$. Then, by the projectivity of $F$, $f$ factors through $F_0Y$.
    In particular, the diagram
    $$\xymatrix{
        P_0(R)\ar[r]^{P_0(f)}\ar[d]^{\theta_0}&P_0(F_0Y)\ar[d]^{\theta_0=\sigma_0}\\
        P_1(R)\ar[r]^{P_1(f)}&P_1(F_0Y)}$$
    commutes since $\theta_0$ is a natural transformation. By the commutativity
    of~\eqref{33}, we see that
    $$\xymatrix{
        P_0(R)\ar[r]^{P_0(f)}\ar[d]^{\theta_0}&P_0(F_iY)\ar[d]^{\sigma_i}\\
        P_1(R)\ar[r]^{P_1(f)}&P_1(F_iY)}$$
    commutes for all $i$. Since $f$ factors as well through
    $\bigoplus_iF_iY\rightarrow Y$, the diagram
    $$\xymatrix{
        P_0(R)\ar[r]^{P_0(f)}\ar[d]^{\theta_0}&\bigoplus_iP_0(F_iY)\ar[d]^{\bigoplus\sigma_i}\ar[r]&P_0(Y)\ar[d]^{\sigma_Y}\\
        P_1(R)\ar[r]^{P_1(f)}&\bigoplus_iP_1(F_iY)\ar[r]&P_1(Y)
    }$$ commutes. This is what we wanted to show.
\end{proof}

\small
\bibliographystyle{amsplain}
\bibliography{ue}

\vspace{20pt}
\scriptsize
\noindent
Benjamin Antieau\\
Northwestern University\\
Department of Mathematics\\
2033 Sheridan Road\\
Evanston, IL 60208\\
\texttt{antieau@northwestern.edu}

\end{document}